\documentclass[]{article}

\usepackage{amsthm}
\usepackage{amsmath}
\usepackage{amssymb}
\usepackage{graphicx}
\usepackage{subfigure}
\usepackage{tikz}
\usepackage{enumerate}
\usetikzlibrary{arrows}
\usepackage{url}
\usepackage{authblk}

\newcommand{\RR}{\mathbb{R}}
\newcommand{\SYSTEMWR}{(\ref{eq:general_system})-(\ref{eq:reset_condition})}
\newcommand{\SYSTEMWRLinear}{(\ref{eq:general_system})-(\ref{eq:reset_condition})-(\ref{eq:linear_system})}
\newcommand{\s}{\mathfrak{s}}

\newcommand{\bx}{\bar{x}}
\newcommand{\R}{{\mathcal{R}}}
\newcommand{\LL}{{\mathcal{L}}}
\newcommand{\Firstpaper}{\cite{GraKruCle13}}
\newcommand{\hphi}{\hat{\varphi}}
\newcommand{\hdelta}{\hat{\delta}}
\newcommand{\conds}{H.1-H.2}

\renewcommand*\thefigure{\thesection.\arabic{figure}}

\makeatletter
\renewcommand\p@subfigure{\thefigure}
\makeatother
\usepackage{xcolor}
\usepackage[normalem]{ulem}

\usepackage{cancel}

\newtheorem{lemma}{Lemma}[section]
\newtheorem{proposition}{Proposition}[section]
\newtheorem{remark}{Remark}[section]
\newtheorem{corol}{Corollary}[section]
\newtheorem{definition}{Definition}[section]

\numberwithin{equation}{section}
\numberwithin{definition}{section}
\numberwithin{remark}{section}
\numberwithin{proposition}{section}
\numberwithin{lemma}{section}
\numberwithin{figure}{section}

\graphicspath{
{./}
{figures/}
{figures/dose/}
{figures/d-invA/T1d9/}
{figures/freq_response/}
{figures/impacts_pperiod/}
{figures/impacts_pperiod_Acorrection/}
{figures/impacts_pperiodT/}
{figures/impacts_pperiodT_Acorrection/}
{figures/bif_R/}
{figures/bif_L/}
{figures/T-periodic/}
{figures/bif1_R/}
{figures/bif1_L/}
{figures/firing-rate_typical/}
{figures/p-orbits/T-per_R-bif/}
{figures/p-orbits/T-per_L-bif/}
{figures/maps/}
{figures/boundary/}
{figures/map_boundary/}
{figures/firing-rate_bell-shape/}
}
\begin{document}
\title{Firing-rate, symbolic dynamics and frequency dependence in
periodically driven spiking models: a piecewise-smooth
approach\thanks{This work has been financially supported by the Large
Scale Initiative Action REGATE,  by the Spanish MINECO-FEDER Grants
MTM2009-06973, MTM2012-31714 and the Catalan Grant 2009SGR859.}}
 
\author[1]{Albert Granados\thanks{albert.granados@inria.fr}}
\author[2]{Maciej Krupa\thanks{maciej.krupa@inria.fr}}
\affil[1,2]{Inria Paris-Rocquencourt}
\date{}
\maketitle

\begin{abstract}
In this work we consider a periodically forced generic
integrate-and-fire model with a unique attracting equilibrium in the
subthreshold dynamics and study the dependence of the firing-rate on
the frequency of the drive.  In an earlier study we have obtained
rigorous results on the bifurcation structure in such systems, with
emphasis on the relation between the firing-rate and the rotation
number of the existing periodic orbits.  In this work we study how
these bifurcation structures behave upon variation of the frequency of
the input. This allows us to show that the dependence of the
firing-rate on frequency of the drive follows a devil's staircase with
non-monotonic steps and that there is an optimal response in the whole
frequency domain.  We also characterize certain bounded frequency
windows in which the firing-rate exhibits a bell-shaped envelope with
a global maximum.
\end{abstract}

\unitlength=\textwidth
\section{Introduction}\label{sec:intro}
In this work we study periodically driven excitable systems of integrate and
fire type, widely used to model the dynamics of the membrane potential
of a neuron. We assume that the periodic
forcing term, or the external input, satisfies a constraint which we refer to as
dose conservation. The constraint is defined as fixing the total amount
(cumulated dose) in a given time (observation time). As argued
in~\S\ref{sec:system_description}, this is equivalent to fixing the average rate of the
cell stimulus, for example the applied current, the amount of neurotransmitter,
or the amount of hormone per time unit, respectively. One of the goals of this
work is to prove that the system exhibits an optimal response, in terms of the
firing-rate, which can be achieved by tuning the system parameters, like period
or amplitude, to obtain the maximal firing-rate of the system (the average
number of spikes per unit time).  In particular, we consider square wave input
and focus on the variation of the period while either the amplitude or the
duration of the pulse is fixed.\\

The mathematical content of our study is to investigate the
bifurcation structure of periodic orbits, as they completely determine
the dynamics for the class of systems we study. In contrast to other
studies~\cite{KeeHopRin81,CooBre99,CooOsb00,Coo01,CooOweSmi01,CooThuWed12,TouBre08,TouBre09,LaiCoo05,JimMihBroNieRub13},
which use Poincar\'e maps, our approach is by means of a stroboscopic
map. This map is discontinuous, but, for most parameter values, it has
the advantage of being contracting on the continuous components.
Hence, as shown in~\Firstpaper{}, results in non-smooth systems can be
applied to get a complete description of periodic orbits and their
rotation numbers. In this work we use this information to understand
the behaviour of the firing-rate under frequency variation of the
drive. In particular, we prove the existence of an optimal response
corresponding to the maximal firing-rate.

The setting we have chosen for this paper is very simple from the
biological point of view, but it has the advantage of being
mathematically tractable.  This is mainly given by assuming that the
unforced system possesses a unique attracting point in the
subthreshold dynamics. Even simple generalizations, for example
allowing the system to undergo a subthreshold saddle-node bifurcation,
lead to complications, as the stroboscopic map can be expansive , so
that the existence of a globally stable attractor cannot be expected.
In particular it is not clear if the firing-rate can be uniquely
defined in the context of such generalizations, as well as how to
obtain rigorous results about it.

The main result of this paper is a complete description of the
response of the system, in terms of the firing-rate, to frequency
variation. In particular, we prove that the firing-rate is maximal for
a certain frequency which depends on the features of the stimulus
(amplitude and duty cycle) as well as on the dynamical properties of
the system. In addition, we provide detailed information on how to
compute such frequency and the corresponding maximal value of the
firing-rate.\\

\noindent This work is organized as follows.\\
In \S\ref{sec:system_description} we describe the integrate-and-fire
system, provide some definitions and state our results. In
\S\ref{sec:bif_scenario} we describe a bifurcation scenario
established in our earlier work~\Firstpaper{}, which we use to prove
our results. In~\S\ref{sec:freq_properties} we describe how the
bifurcations in this parameter space change under frequency variation
of the input.  We also provide a precise statement of our results and
their proofs.  In~\S\ref{sec:optimization} we present a result
regarding the optimization of the firing-rate in terms of frequency of
the input. Finally, in \S\ref{sec:examples} we apply these results to
an example, a linear integrate-and-fire neuron (LIF), to completely
describe the firing-rate response under frequency variation.

\section{The model, definitions and statement of results}\label{sec:system_description}
In the context of neuronal modeling or hormone segregation one
relies on excitable systems, which are able to exhibit certain responses
given by large amplitude oscillations (spikes) as a response to
certain stimulation. One of the most extended type of systems
exhibiting this behavior are hybrid systems (a generalization of the
so-called integrate-and-fire systems)
which can be seen as approximation of
slow/fast systems. That is, systems of the form
\begin{equation}
\dot{x}=f(x)+I(t),\,x\in\RR,
\label{eq:general_system}
\end{equation}
where $x$ represents an action potential or the output of the cell and
$I(t)$ an external stimulation, which could be the output of another
cell. Then, system~\eqref{eq:general_system} is submitted to the
reset condition
\begin{equation}
x=\theta \longrightarrow x=0,
\label{eq:reset_condition}
\end{equation}
that is, the trajectories of system~\eqref{eq:general_system} are
instantaneously reset to $0$ whenever they reach the threshold given by
$x=\theta$. Due to this instantaneous reset the solutions of the
system exhibit discontinuities which emulate the spikes. In this work,
we will consider $\theta$ a constant, although it is a common approach
to add certain dynamics to this threshold in order to model more
complex behaviours, as type III excitability~\cite{MenHugRin12}.

As mentioned in the introduction, we will assume in this paper that
the cell's input, $I(t)$, consists of a $T$-periodic square-wave
function, 
\begin{equation}
I(t)=\left\{
\begin{aligned}
&A&&\text{if }t\in\left(nT,nT+dT\right]\\
&0&&\text{if }t\in(nT+dT,(n+1)T],\\
\end{aligned}\right.
\label{eq:pulse}
\end{equation}
This is a well accepted, both in neuroscience and
neuroendocrinology, to assume that inputs to excitable cells are given
by functions of this form, as they occur in a pulsatile way. Other
works consider rectified sinusoidals as inputs to neurons in the
auditory brainstem~\cite{MenHugRin12}.\\
The square wave function $I(t)$ will be characterized by three parameters: its
amplitude $A$, its period $T$ and the \emph{duty cycle} $d$, which is the
duration of the pulse with respect to $T$.\\

As mentioned in~\S\ref{sec:intro}, the main goal of this work is to
study the response of the system in terms of the \emph{firing-rate}
(number of spikes per unite time). In particular, we are interested on
its optimization under the variation of parameters $A$, $d$ and $T$.
However, we impose a constraint that the total amount of the released
quantity be constant per stimulation period. We will refer to this as
\emph{dose conservation}, with the following biological question in
mind: given a certain available quantity, how does it have to be
released to the excitable cell in order to obtain from it the highest
firing-rate?  Assuming that the experimental observation time, $\tau$,
is large enough relative to the different periods of the signal
$I(t)$, $T$, the total amount of released quantity can be approximated
by
\begin{equation*}
\int_0^\tau I(t)dt\sim Q\tau,
\end{equation*}
where $Q$ is the average value of $I$ over one period,
\begin{equation}\label{eq:dose}
Q:=\frac{1}{T}\int_0^TI(t)dt,
\end{equation}
which we will call \emph{dose}. Therefore, the cumulative dose
released to the cell will be maintained as long as the dose $Q$ is
conserved.\\
In order to add the dose conservation to our system, one has only to
keep constant the product $Ad$, which can be performed in different
ways. In this work we will focus on two of them, the trivial one by
keeping constant both $A$ and $d$ (width correction) and also the
other one varying both $A$ and $d$ so that the total duration of the
pulse, $\Delta=dT$, is constant (amplitude correction).\\
In section~\ref{sec:bifurcation_analysis} we will obtain theoretical
results for the first case, which will be used also to study the
behavior of the firing-rate under frequency variation for the second
case in an example in~\S\ref{sec:amplitude_correction}.

As mentioned above, the reset condition~\eqref{eq:reset_condition}
introduces discontinuities to the solutions of the system. However,
despite these discontinuities, the solutions of the non-autonomous
system~\SYSTEMWR{} are well defined. Let $\phi(t;t_0,x_0)$ be the
solution of system~\SYSTEMWR{} fulfilling $\phi(t_0;t_0,x_0)=x_0$. As
usual in piecewise-smooth systems, the flow $\phi$ is obtained by
properly matching the solutions for $0<t \mod T\le dT$ and $dT<t\mod
T\le T$ combined with the reset
condition~\eqref{eq:reset_condition}. This makes the flow $\phi$
non-differentiable at $t\mod T=dT$ and $t\mod T=T$ and discontinuous
at the spikes times, those at which the threshold is reached.
\begin{remark}
As we are interested in periodic orbits, although system~\SYSTEMWR{} is
non-autonomous, we will assume from now on that $t_0=0$, and we will write
$\phi(t;x_0)$. Note that if $t_0\ne0$, the initial condition for a periodic
orbit (fixed point of the stroboscopic map) may be different, although it 
still exists.
\end{remark}

Let us assume that the system
\begin{equation}
\dot{x}=f(x)
\label{eq:autonomous_system}
\end{equation}
satisfies the following conditions.
\begin{enumerate}[H.1]
\item \eqref{eq:autonomous_system} possesses an attracting equilibrium point
\begin{equation}
0<\bx<\theta,
\label{eq:critical_point}
\end{equation}
\item $f(x)$ is monotonic decreasing function in $[0,\theta]$:
\begin{equation*}
f'(x)<0,\;0\le x\le \theta.
\end{equation*}
\end{enumerate}

As shown in~\Firstpaper{}, system~\SYSTEMWR{} possesses attracting
periodic orbits for almost all (except in a cantor set with zero
measure) values of $A$, $\theta$ and $d$ as long as
conditions~\conds{} are satisfied and $T$ is large or small enough.
These periodic orbits may be continuous (subthreshold dynamics) or
discontinuous (spiking dynamics).  Let $\phi(t;x_0)$, with
$\phi(0;x_0)=x_0$, be an orbit of the non-autonomous
system~\SYSTEMWR{}.  Then we consider
\begin{equation}\label{eq:defirate}
r(x_0)=\lim_{\tau\to\infty}\frac{\#(\mbox{spikes performed by } \phi(t;x_0)
\mbox{ for } t\in[0,\tau])}{\tau},
\end{equation}
where \# means \emph{number of}, if this limit exists. We then define the
\emph{firing-rate}.
\begin{definition}\label{def:firing-rate}
If $r(x_0)$ does not depend on $x_0$ then we
call it $r$, the firing-rate.
\end{definition}
The firing-rate can be seen as the average number of spikes per unit
time performed by the system along a periodic orbit.\\
Note that the firing-rate is well defined whenever there exists a unique
attracting periodic orbit. However, it will in general depend on the system
parameters $T$, $A$ and $d$.

Unlike in other approaches
(\cite{KeeHopRin81,CooBre99,CooOsb00,CooOweSmi01,CooThuWed12,TouBre08}),
in order to study integrate-and-fire model~\SYSTEMWR{} our essential
tool will be the stroboscopic map. Given an initial condition $x_0$,
this map consists in flowing the system~\SYSTEMWR{} for a time $T$,
the period of the drive, and is the usual tool used when dealing with
(smooth) periodic non-autonomous systems. In other words, it becomes
\begin{equation}
\s(x_0)=\phi(T;x_0),
\label{eq:stroboscopic_map}
\end{equation}
where $\phi(t;x_0)$ is the flow associated with~\SYSTEMWR{}. In the
mentioned works, authors considered a Poincar\'e map from the threshold
to itself (when spikes occur), added time as a variable and studied
the times given by the spikes.\\
As it will be detailed below in~\S\ref{sec:bif_scenario}, the
stroboscopic map will be piecewise-defined and discontinuous, and
hence it is typically avoided in periodically forced hybrid systems,
as one cannot apply classical results for regular smooth systems.
These discontinuities of the map will not be given by the spikes
performed by the trajectories of the system. On the contrary, the
stroboscopic will undergo a discontinuity at those initial conditions
for which the number of spikes performed by the trajectories for
$t\in[0,T]$ changes (see Fig.~\ref{fig:boundary}).  Despite these
discontinuities, using results for non-smooth systems, the dynamics of
the map is completely understood (see~\Firstpaper{} for a discussion
and references). This includes the rotation number, also called
winding number, $\rho$, of all possible periodic orbits of the
stroboscopic map, which will be of special interest in our work.  The
rotation number is usually associated with circle maps and,
intuitively, measures the average rotation along trajectories when it
does not depend on its initial condition. Under certain conditions,
discontinuous piecewise-defined maps can be reduced to circle maps,
and, when a periodic orbit exists, the rotation number becomes the
ratio between the number of steps at the right of the discontinuity of
the map along the periodic orbits to its period
(see~\cite{AlsGamGraKru14} for more details).

As shown in~\cite{KeeHopRin81} (see also~\Firstpaper{} and
section~\S\ref{sec:bif_scenario} below), the rotation number of
the periodic orbits is well related with the number of spikes performed at each
period of a periodic orbit of the stroboscopic map. A crucial quantity will be
the average number of spikes by period of the stroboscopic map, which was named
in~\cite{KeeHopRin81} \emph{firing-number}, $\eta$. This is given more
precisely by the following definition.
\begin{definition}\label{def:firing_number}
Let $n$ be the total number of spikes performed by a $p$-periodic orbit of the
stroboscopic map $\s(x)$, $n,p\in\mathbb{N}$; then we define the 
firing-number as
\begin{equation}
\eta=\frac{n}{p},
\label{eq:average_spkes_pT}
\end{equation}
which is the average number of spikes per iteration of the stroboscopic map
along a periodic orbit.
\end{definition}
\begin{remark}
Then, assuming that the mentioned periodic orbit is attracting, the firing-rate
can be obtained from the firing-number as
\begin{equation}
r=\frac{\eta}{T}.
\label{eq:firing_rate}
\end{equation}
\end{remark}
As will be shown in~\S\ref{sec:freq_properties}
(Corollary~\ref{cor:regions}), depending on the value of the dose $Q$
defined in~\eqref{eq:dose} the firing-rate will exhibit qualitatively
different behaviors. This will bring us to consider a \emph{critical
dose}, which we define as follows.
\begin{definition}
The critical dose, $Q_c$, is the value of $A>0$ that places the equilibrium
point, $\bx$, of the system $\dot{x}=f(x)+A$ at the threshold; it is given by
\begin{equation}
f(\theta)+Q_c=0.
\label{eq:critical_dose_eq}
\end{equation}
\end{definition}
Note that $Q_c$ is the minimal dose that permits the system~\SYSTEMWR{} to
exhibit spikes when it is driven constantly, $I(t)=Q_c$ ($d=1$ and $A=Q_c$).

Our goal is to study the qualitative behavior of the firing-rate under
variation of the period of the input, $T$, for a chosen $Q$. We then prove the
following results when $A>0$ and $d\in(0,1)$ are kept constant (width correction
for dose conservation).
\begin{enumerate}
\item The firing-rate, obtained as the ratio of the firing-number
$\eta$ to $T$, follows a devil's staircase with monotonically
decreasing steps (see Figure~\ref{fig:fr_typical}) (except possibly in
a compact set of values of $T$). This is a consequence of
Propositions~\ref{prop:Tlarge} and~\ref{prop:Tsmall}
(Corollary~\ref{coro:devils_staircase}).
\item The firing-rate for low frequency inputs fulfills
\begin{equation*}
\lim_{T\to\infty}r(T)=\frac{d}{\delta},
\end{equation*}
where $\delta$ is the time needed by system $\dot{x}=f(x)+A$ to reach the
threshold from $x=0$. This is Proposition~\ref{prop:fr_Tlarge}.
\item If $I(t)$ is such that $Q=Ad<Q_c$, then the firing-rate becomes
zero for large enough frequencies. This is also a consequence of
Propositions~\ref{prop:Tlarge} and~\ref{prop:Tsmall}
(Corollary~\ref{cor:regions}).
\item If $Q>Q_c$, then
\begin{equation*}
\lim_{T\to0}r(T)=\frac{1}{\hdelta},
\end{equation*}
where $\hdelta>0$ is the time needed for the averaged system $\dot{x}=f(x)+Ad$ to
reach the threshold from $x=0$. This is Proposition~\ref{prop:fr_Tsmall}.
\item The firing-rate exhibits a global maximum and minimum in
$T\in(0,\infty)$.  Let $0<T_1<T_2$ such that $\eta(T)=1$ for
$T\in[T_1,T_2]$. Then, if $T_1$ is large enough, the maximal firing
rate occurs for $T=T_1$. The minimal one corresponds to the minimum
between $0$, $1/\hdelta$ and $1/T_2$. This is
Proposition~\ref{prop:optimization} and
Remark~\ref{rem:global_minimum}.
\end{enumerate}

\begin{figure}
\begin{center}
\includegraphics[width=0.4\textwidth,angle=-90]{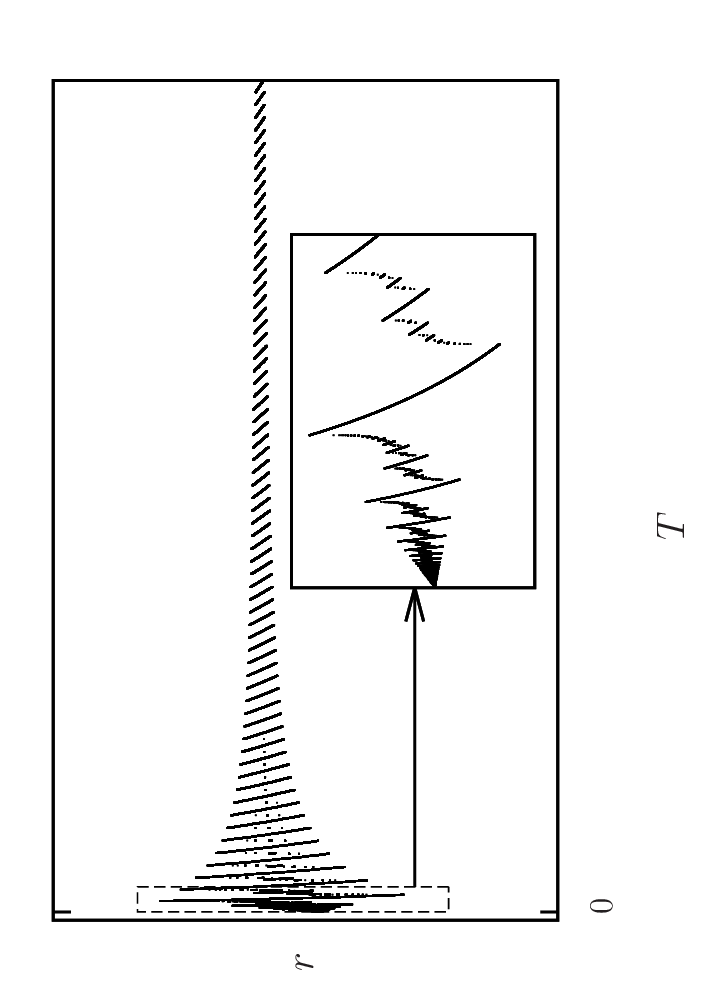}
\end{center}
\caption{Typical response of the firing-rate upon frequency variation.
It is given by a fractal structure of monotonically decreasing pieces,
following a devil's staircase. The width of the ``main pieces''
stabilizes at $O(\delta/d)$, and each exhibits a local maximum and
minimum. The firing-rate has limiting value $d/\delta$ when
$T\to\infty$, and $1/\hdelta$ when $T\to0$, and it exhibits a global
maximum and minimum. See text (Section~\ref{sec:system_description}) for the
definition of these parameters.}
\label{fig:fr_typical}
\end{figure}

\section{Bifurcation analysis}\label{sec:bifurcation_analysis}
\subsection{The two-dimensional parameter space}\label{sec:bif_scenario}
In this section we provide a summary of the results shown
in~\Firstpaper{}, see there for the details and proofs of what follows
in this section.\\

As mentioned in~\S\ref{sec:system_description}, due to the periodicity
of $I(t)$, we will use the stroboscopic
map~\eqref{eq:stroboscopic_map}, which is a discontinuous
piecewise-smooth map, in order to understand the dynamics of
system~\SYSTEMWR{}. This map is a smooth map (as regular as
$f(x)$~\eqref{eq:general_system}) in certain regions in the state
space $[0,\theta]$ characterized by the number of spikes performed by
$\phi$, the discontinuous flow associated with system~\SYSTEMWR{},
when flowed for a time $T$. This is because, in these regions, the
stroboscopic map becomes a composition of maps obtained by integrating
system~\eqref{eq:general_system} and reseting from $x=\theta$ to
$x=0$. Both types of intermediate maps are smooth. These regions in
the state space are separated by boundaries of the form $x=\Sigma_n$,
$\Sigma_n=\Sigma_n(A,T,d)$, where the stroboscopic map is
discontinuous. At the right of $x=\Sigma_n$ the trajectories
of~\SYSTEMWR{} exhibit $n$ spikes when flowed for a time $T$, whereas
at its left they exhibit $n-1$ spikes (see Figure~\ref{fig:boundary}
for $n=3$).\\
As the number of spikes can be arbitrarily large (for $A>0$ large
enough), the state space $[0,\theta]$ is potentially divided in an
infinite number of such regions. However, for fixed parameter values,
the state space is split in at most two regions, $[0,\Sigma_n)$ and
$[\Sigma_n,\theta)$, where the trajectories perform $n-1$ and $n$
spikes, respectively, when flowed during a time $T$. This comes from
the fact that the initial condition that separates different sets of
initial conditions leading to different number of spikes for
$t\in[0,T]$ is unique, as it is the one spiking exactly at $t=dT$. We
refer to~\Firstpaper{} for further details.
\begin{figure}
\begin{center}
\begin{picture}(1,0.5)
\put(0,0.35){
\subfigure[]
{\includegraphics[angle=-90,width=0.5\textwidth]{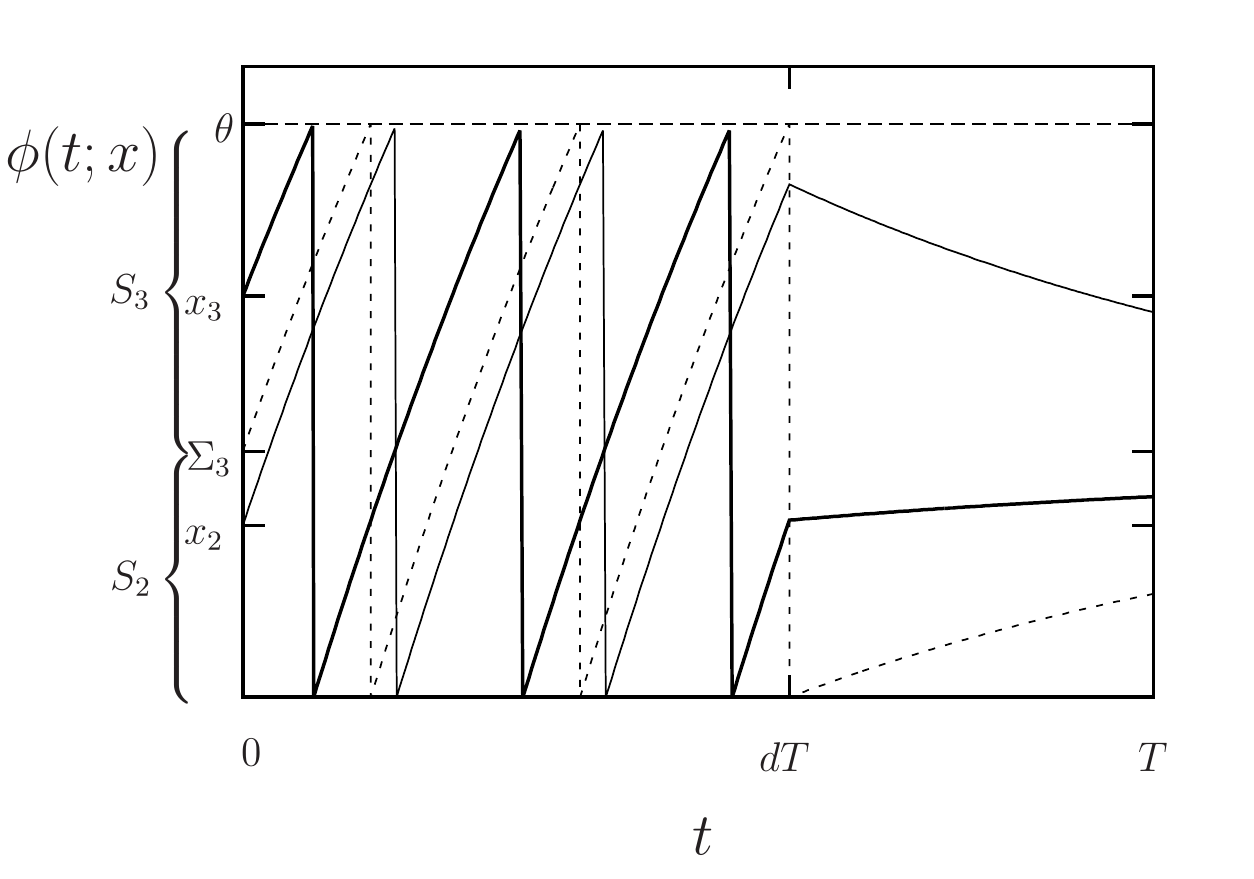}}
}
\put(0.5,0.35){
\subfigure[]
{\includegraphics[angle=-90,width=0.5\textwidth]{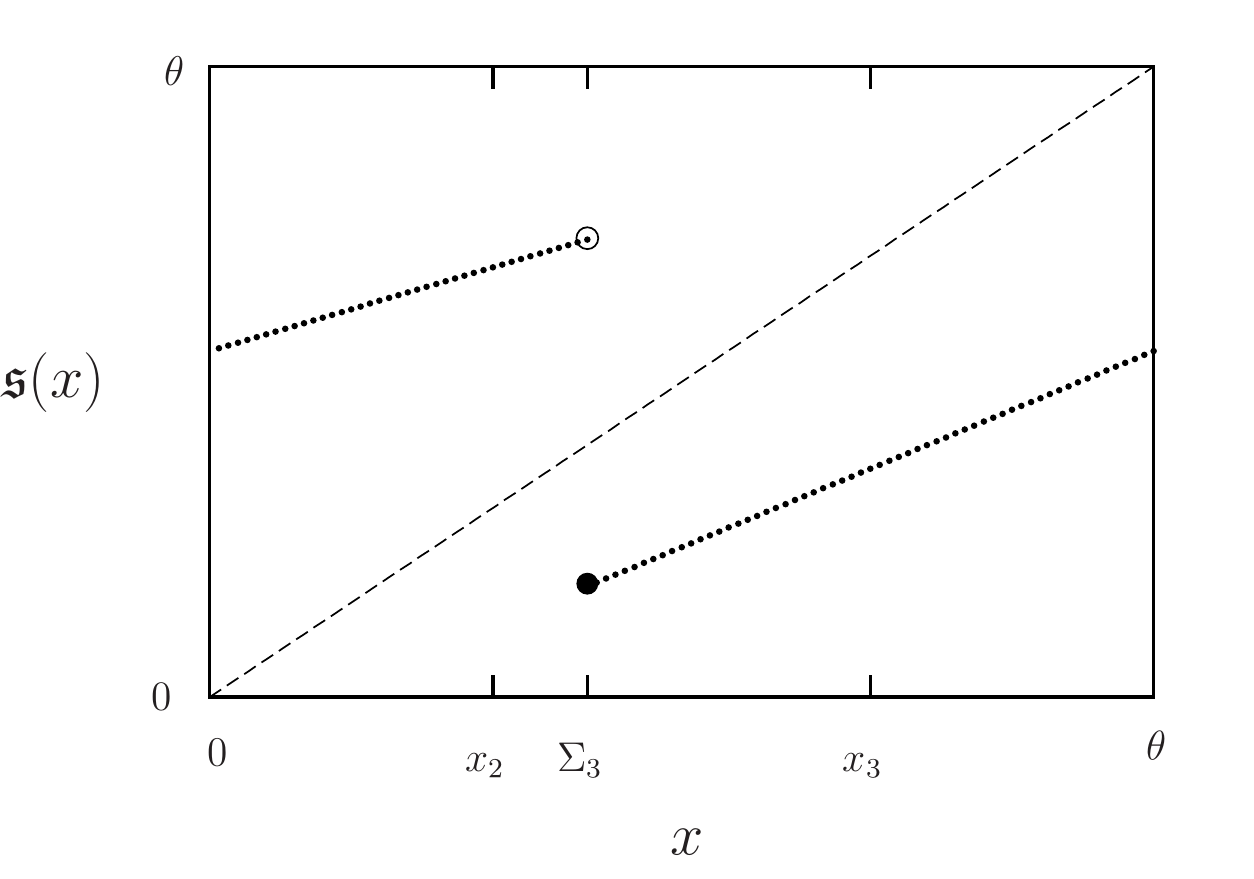}}
}
\end{picture}
\end{center}
\caption{In (a) the trajectories of systems~\SYSTEMWR{}. Dashed line:
trajectory with $\Sigma_3$ as initial condition. Solid thick line:
trajectory with $x_3>\Sigma_3$ as initial condition, which spikes $3$
times.  Solid thin line: trajectory with $x_2<\Sigma_3$ as initial
condition, which spikes $2$ times. In (b) the stroboscopic map, with a
discontinuity at $x=\Sigma_3$.}
\label{fig:boundary}
\end{figure}

The possible dynamics of the stroboscopic map, and hence of
system~\SYSTEMWR{}, is completely captured in the two-dimensional
parameter space $d\times 1/A$.  Thus, by understanding the bifurcation
structures in this parameter space one obtains a complete description
of the fixed points, periodic orbits, their rotation numbers and their
firing-rate.\\
Under the assumptions~\conds{}, and if $T$ is small or large enough,
the bifurcation scenario in the parameter space given by $d\times 1/A$
for $T>0$ is equivalent to the one shown in
Figure~\ref{fig:d-invA_generic}, which is described below and
rigorously proven in~\Firstpaper{}.
\begin{figure}
\begin{center}
\begin{picture}(1,0.5)
\put(0,0.35){
\subfigure[\label{fig:d-invA_generic_regions}]
{\includegraphics[angle=-90,width=0.5\textwidth]{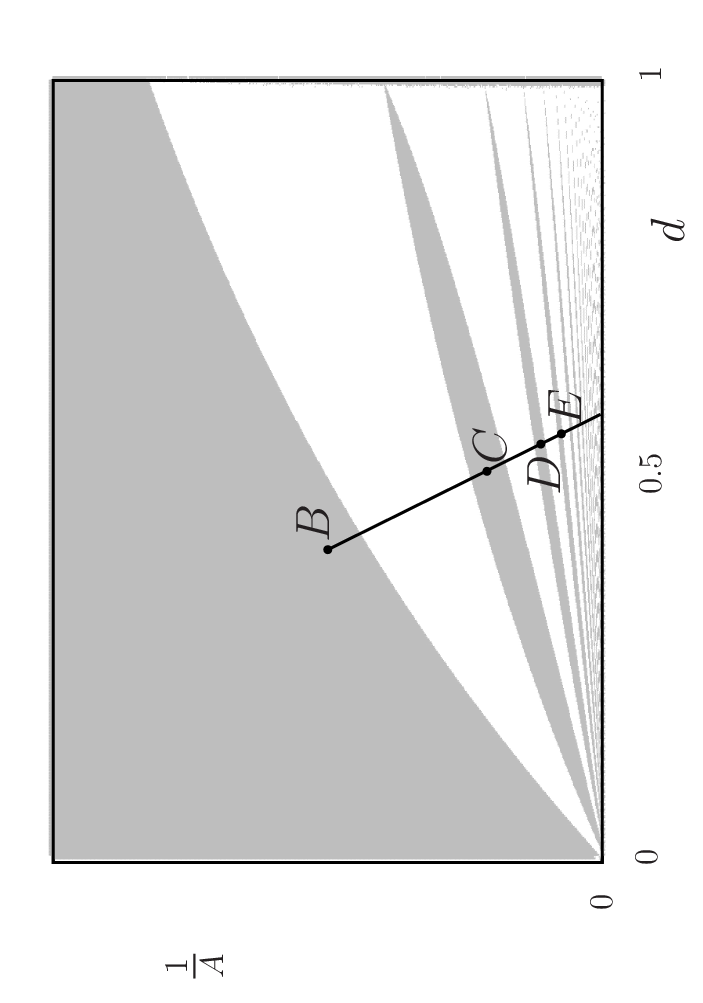}}
}
\put(0.5,0.35){
\subfigure[\label{fig:d-invA_generic_1dscann}]
{\includegraphics[angle=-90,width=0.5\textwidth]{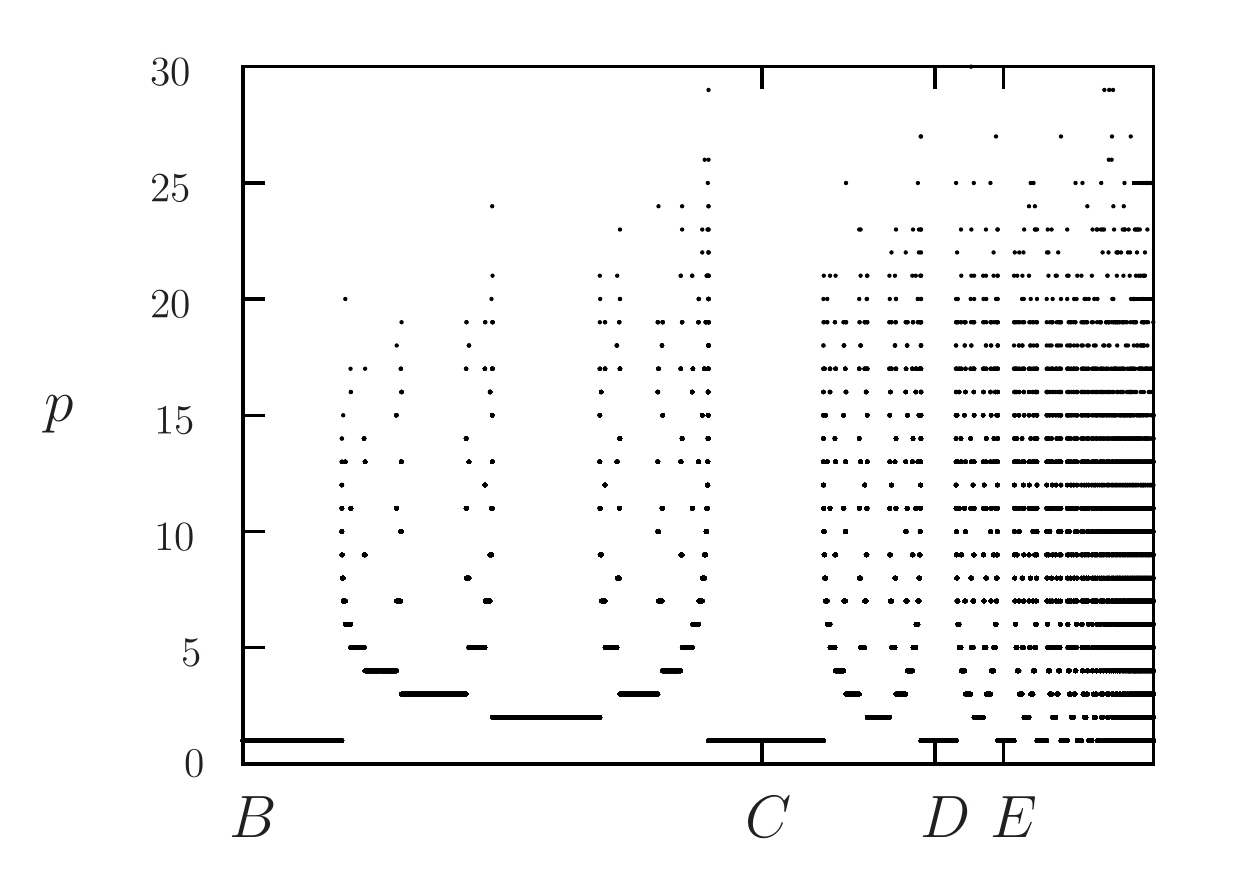}}
}
\end{picture}
\end{center}
\caption{(a) Bifurcation scenario for system~\SYSTEMWR{.} In gray
regions there exist $T$-periodic orbits. In the white regions, one
finds higher periodic orbits following and adding structure (see
text). In $B$, $C$, $D$ and $E$ one finds $T$-periodic orbits spiking
$0$, $1$, $2$ and $3$ times per period, respectively. (b) Periods of
the periodic orbits found along the segment shown in (a).}
\label{fig:d-invA_generic}
\end{figure}
As suggested in Figure~\ref{fig:d-invA_generic_regions}, there exists an
infinite number of regions (in gray) accumulating to the horizontal axis for
which only $T$-periodic orbits spiking $n$ times exist. These are fixed points,
$\bx_n$, of the stroboscopic map $\s(x)$~\eqref{eq:stroboscopic_map}. These
regions in parameter space are ordered, in the clockwise direction, in such a
way that these $T$-periodic orbits spike $0,1,2,3,\dots$ times per period. The
bifurcation curves that bound the gray regions are given by border collision
bifurcation of the map. That is, the fixed points of the stroboscopic map
$\bx_n$ collide with one of the boundaries, $\bx_n=\Sigma_{n}$
(Figures~\ref{fig:T-per_R-bif} and~\ref{fig:map_T_R-bif} for $n=3$) and
$\bx_n=\Sigma_{n+1}$ (Figures~\ref{fig:T-per_L-bif} and~\ref{fig:map_T_L-bif} for
$n=3$), and no longer exist. This defines the upper and lower bifurcation
curves, respectively, bounding each gray region as follows.
\begin{definition}\label{def:AnR_AnL}
For $d\in(0,1)$, we define $A_n^\R(d)$ and $A_n^\LL(d)$, $n\ge 1$, the values of
$A$ for which the fixed point $\bx_n$ collides with the boundaries
$\Sigma_{n}$ and $\Sigma_{n+1}$, respectively:
\begin{align*}
\lim_{A\to \left(A_n^\R\right)^+}\bx_n&=\Sigma_{n}\\
\lim_{A\to \left(A_n^\LL\right)^-}\bx_n&=\Sigma_{n+1}.
\end{align*}
The fixed point $\bx_0\in S_0$ undergoes only one border collision bifurcation,
when it collides with $\Sigma_1$ from the left.  This one occurs for $A=A_0(d)$,
\begin{equation*}
\lim_{A\to \left( A_0 \right)^-}\bx_0=\Sigma_1.
\end{equation*}
\end{definition}
Hence, a fixed point $\bx_n\in [0,\theta]$ will exist if
$A\in[A_n^\R,A_n^\LL)$.\\
\begin{remark}
The values $A_n^{\R,\LL}(d)$ depend also on $T$; we will explicitly specify this
when convinient.
\end{remark}
At the upper bifurcation curves ($A=A_n^\R(d)$), the fixed points collide with a
boundary from its right (Figures~\ref{fig:T-per_R-bif}
and~\ref{fig:map_T_R-bif}), and hence will be associated to the $\R$ symbol.  On
the lower ones ($A=A_n^\LL(d)$) fixed points collide with another boundary from
its left (see Figures~\ref{fig:T-per_L-bif} and~\ref{fig:map_T_L-bif}), and will
have associated the symbol $\LL$. Note that the stroboscopic map fulfills
$\s(\Sigma_n)=\s(\Sigma_n^+)$, and hence the fixed points no longer exist at their
left bifurcations whereas they still do for the right bifurcations (note the
gray shown in Figures~\ref{fig:T-per_L-bif} and~\ref{fig:map_T_L-bif}).
\begin{figure}
\begin{center}
\begin{picture}(1,0.8)
\put(0,0.8){
\subfigure[\label{fig:T-per_R-bif}]
{\includegraphics[angle=-90,width=0.5\textwidth]{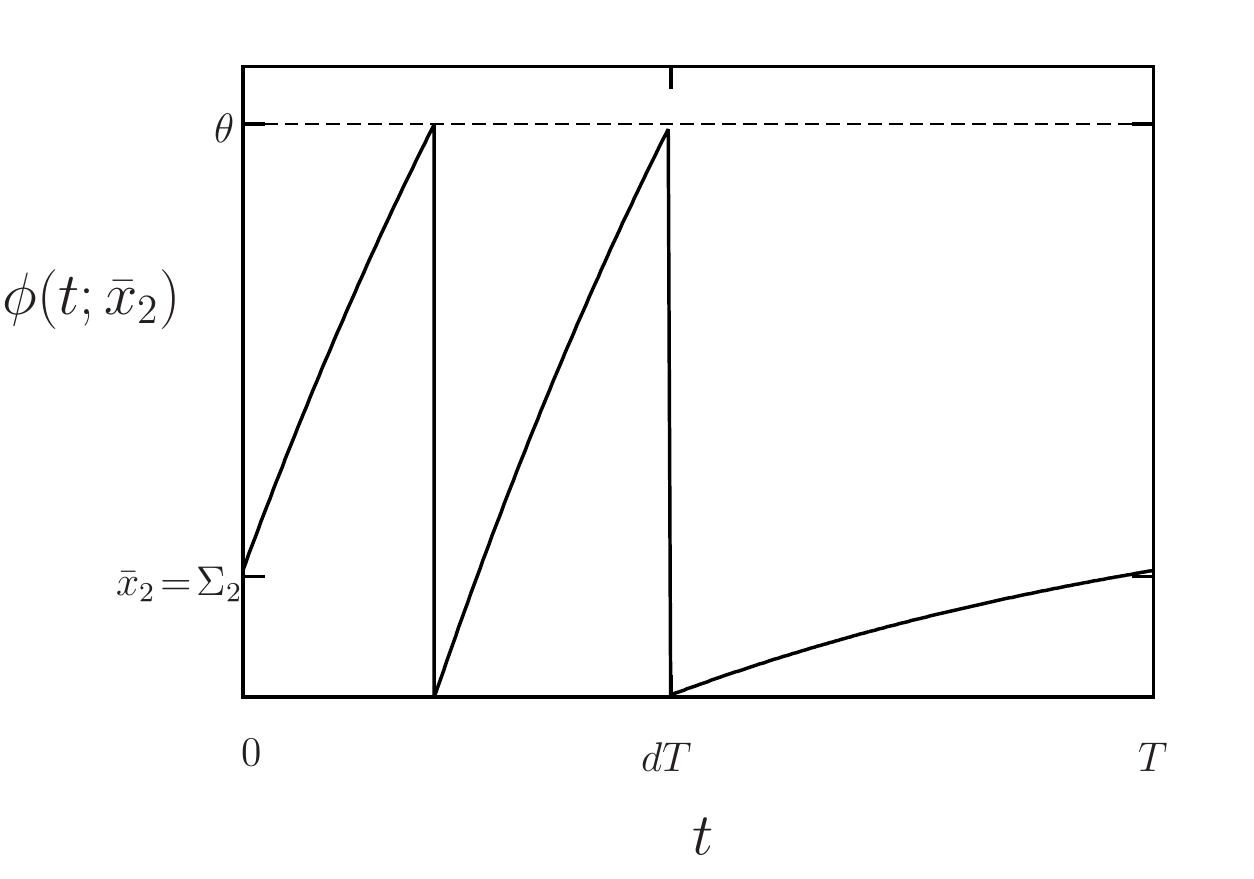}}
}
\put(0.5,0.8){
\subfigure[\label{fig:T-per_2}]
{\includegraphics[angle=-90,width=0.5\textwidth]{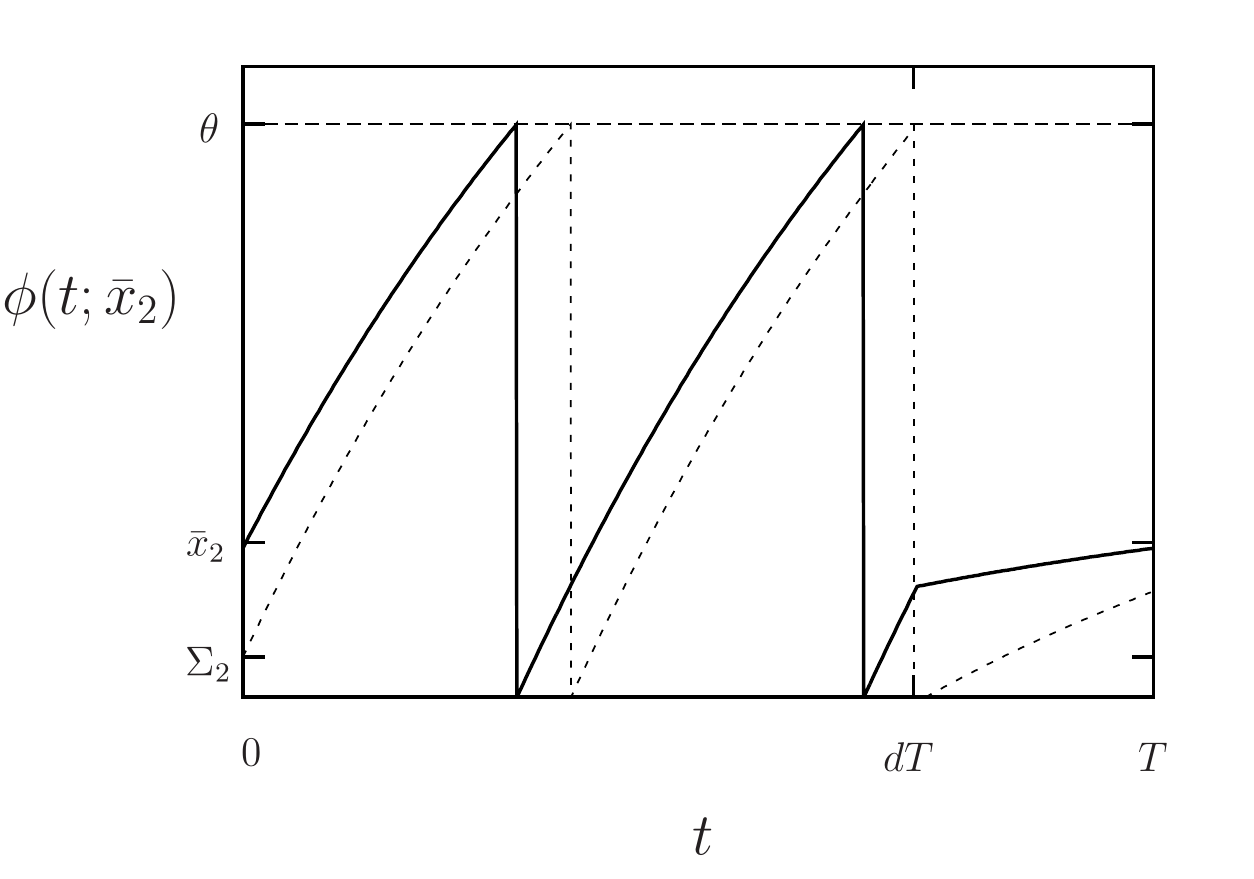}}
}
\put(0,0.4){
\subfigure[\label{fig:T-per_3}]
{\includegraphics[angle=-90,width=0.5\textwidth]{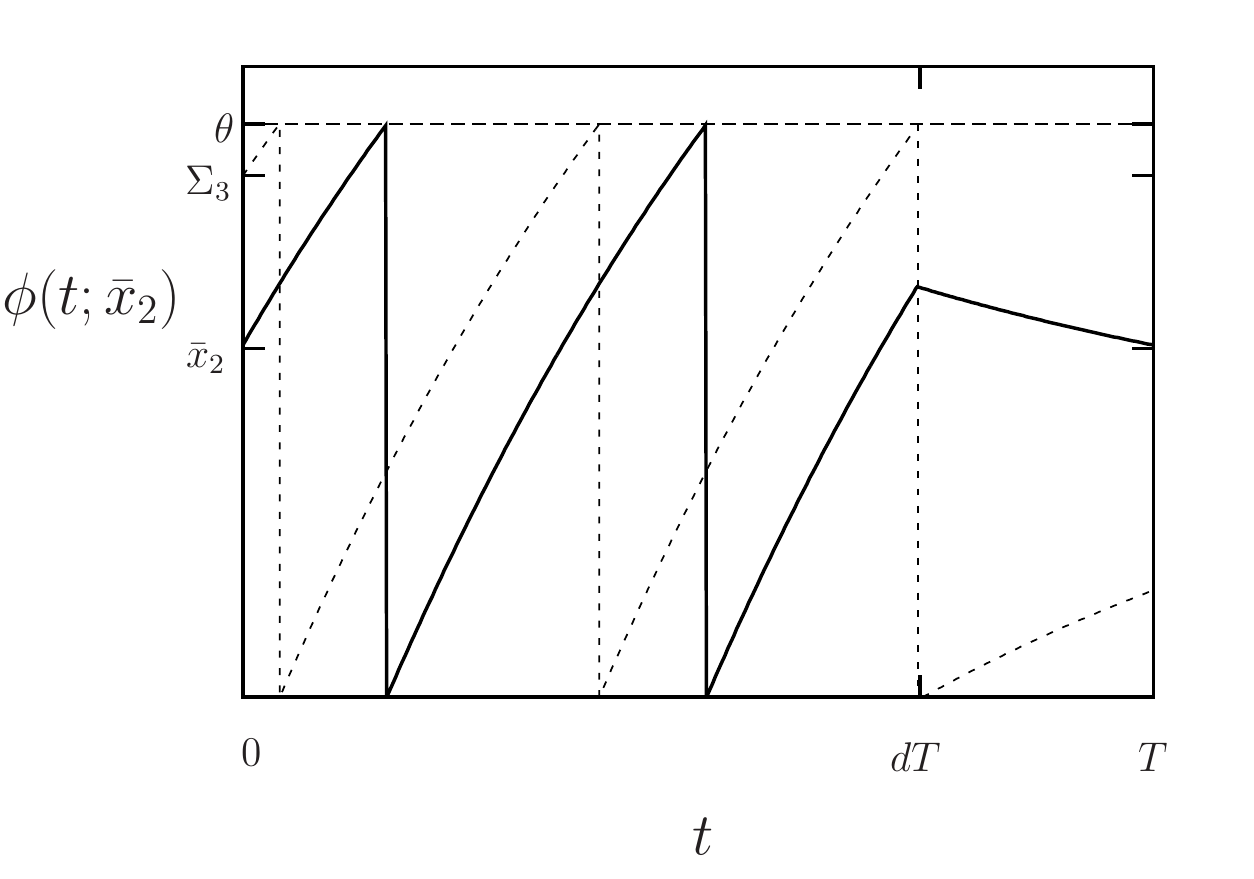}}
}
\put(0.5,0.4){
\subfigure[\label{fig:T-per_L-bif}]
{\includegraphics[angle=-90,width=0.5\textwidth]{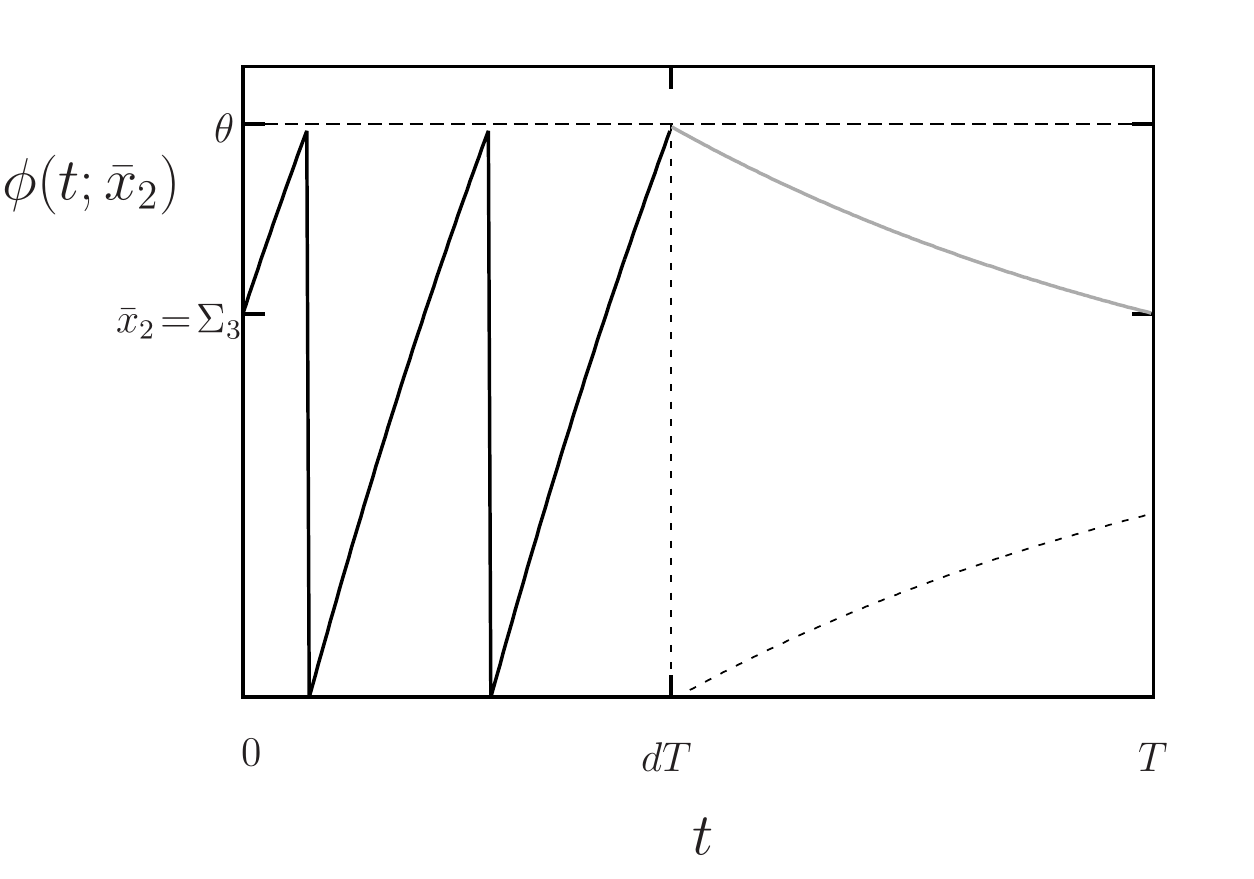}}
}
\end{picture}
\end{center}
\caption{$T$-periodic orbit spiking twice per period (fixed point $\bx_2$ of the
stroboscopic map). It undergoes border collision bifurcations when it collides
with the boundaries $\Sigma_2$ and $\Sigma_3$ (a) and (d), respectively.  The
periodic orbit shown in (d) is its limit when $\bx_2\to\Sigma_3^-$; note that
for $\bx_2=\Sigma_3$ it should be reset to $0$ at $t=dT$, this is why it is
shown in gray. In (b) and (c), the trajectories of these boundaries are shown in
dashed lines; note that they collide with the threshold at $t=dT$.  Parameter
values for panel~(c) are the same as for point $D$ of
Figure~\ref{fig:d-invA_generic_regions}.  The four figures are in one to one
correspondence with the four figures of Figure~\ref{fig:maps_bx2}, where the
stroboscopic map is shown for the same parameter values.}
\label{fig:bif_T-per}
\end{figure}
\begin{figure}
\begin{picture}(1,0.8)
\put(0,0.8){
\subfigure[\label{fig:map_T_R-bif}]{\includegraphics[angle=-90,width=0.5\textwidth]
{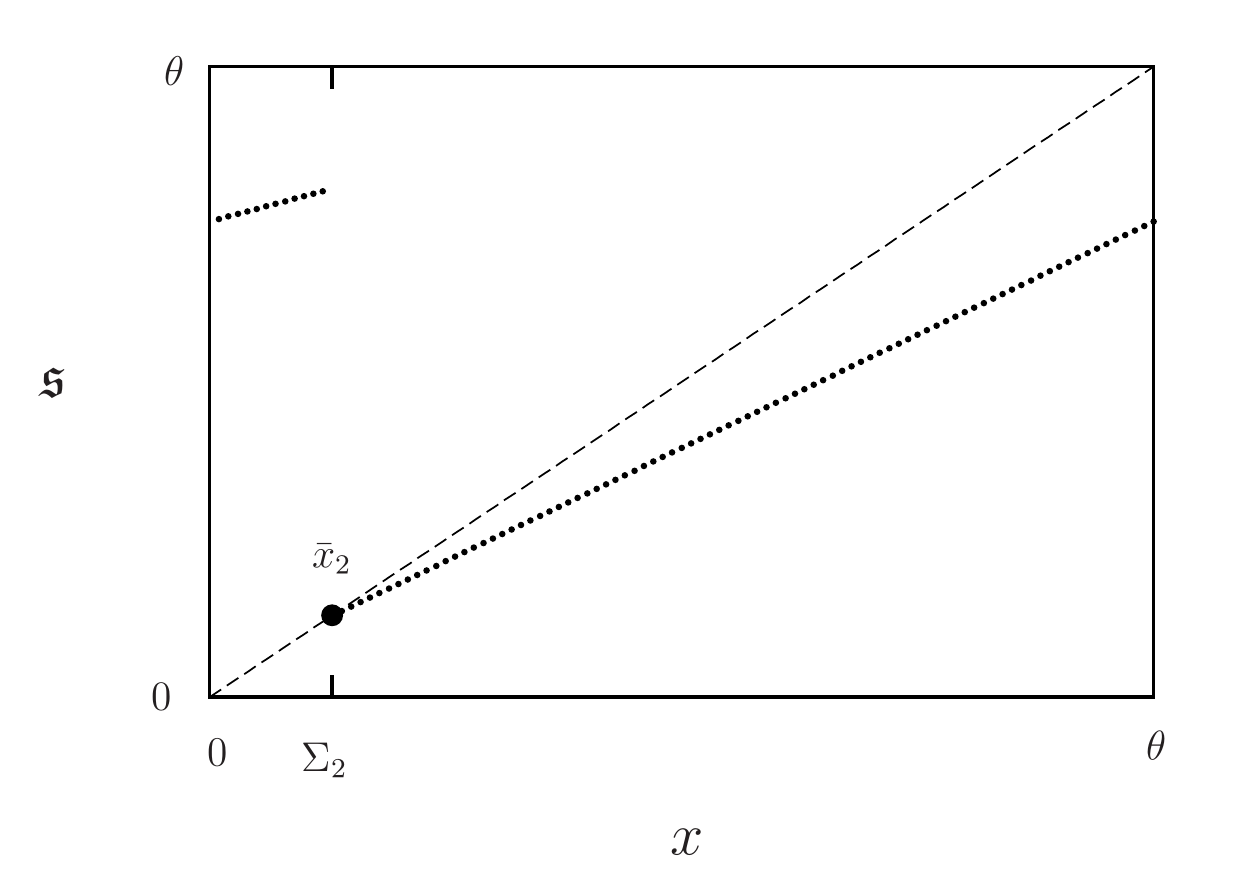}}
}
\put(0.5,0.8){
\subfigure[\label{fig:map2}]{\includegraphics[angle=-90,width=0.5\textwidth]
{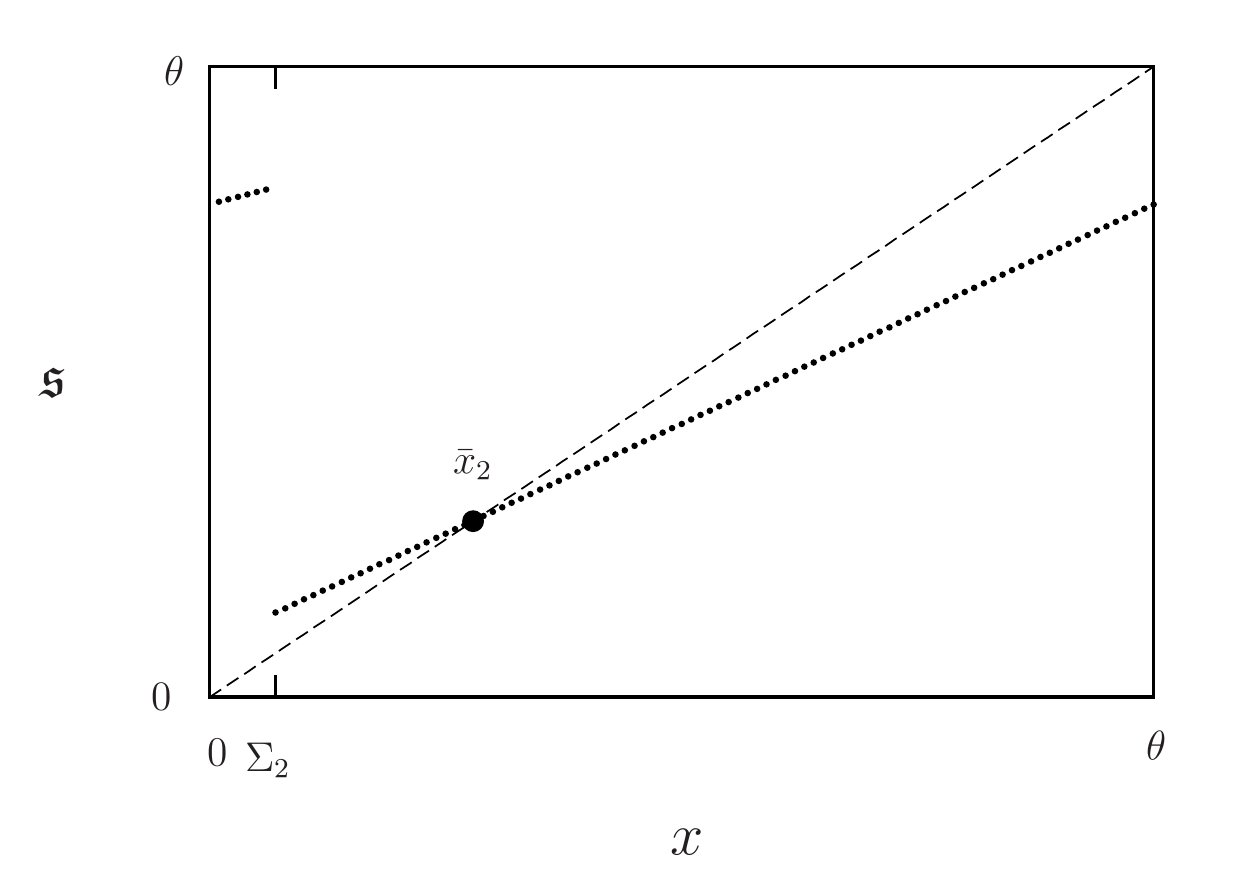}}
}
\put(0,0.4){
\subfigure[\label{fig:map3}]{\includegraphics[angle=-90,width=0.5\textwidth]
{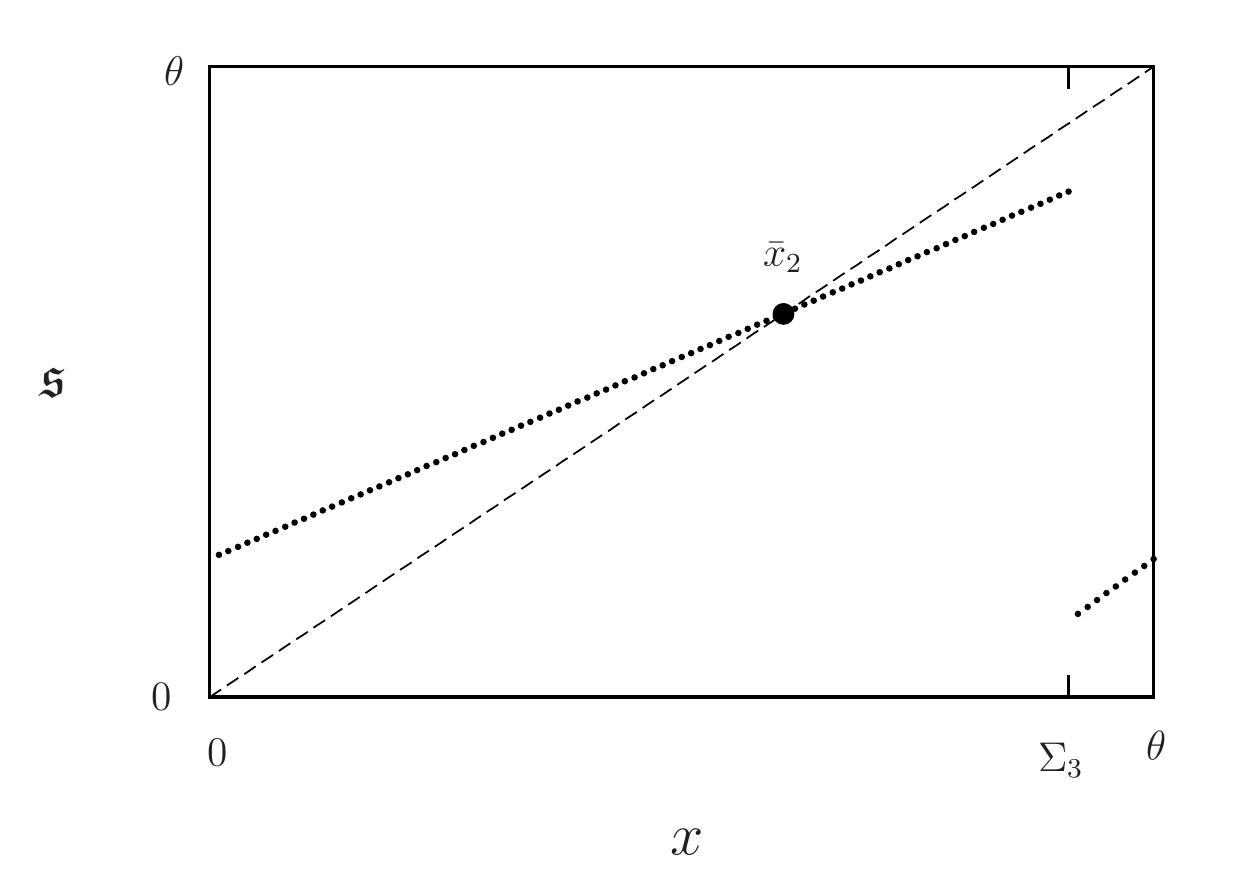}}
}
\put(0.5,0.4){
\subfigure[\label{fig:map_T_L-bif}]{\includegraphics[angle=-90,width=0.5\textwidth]
{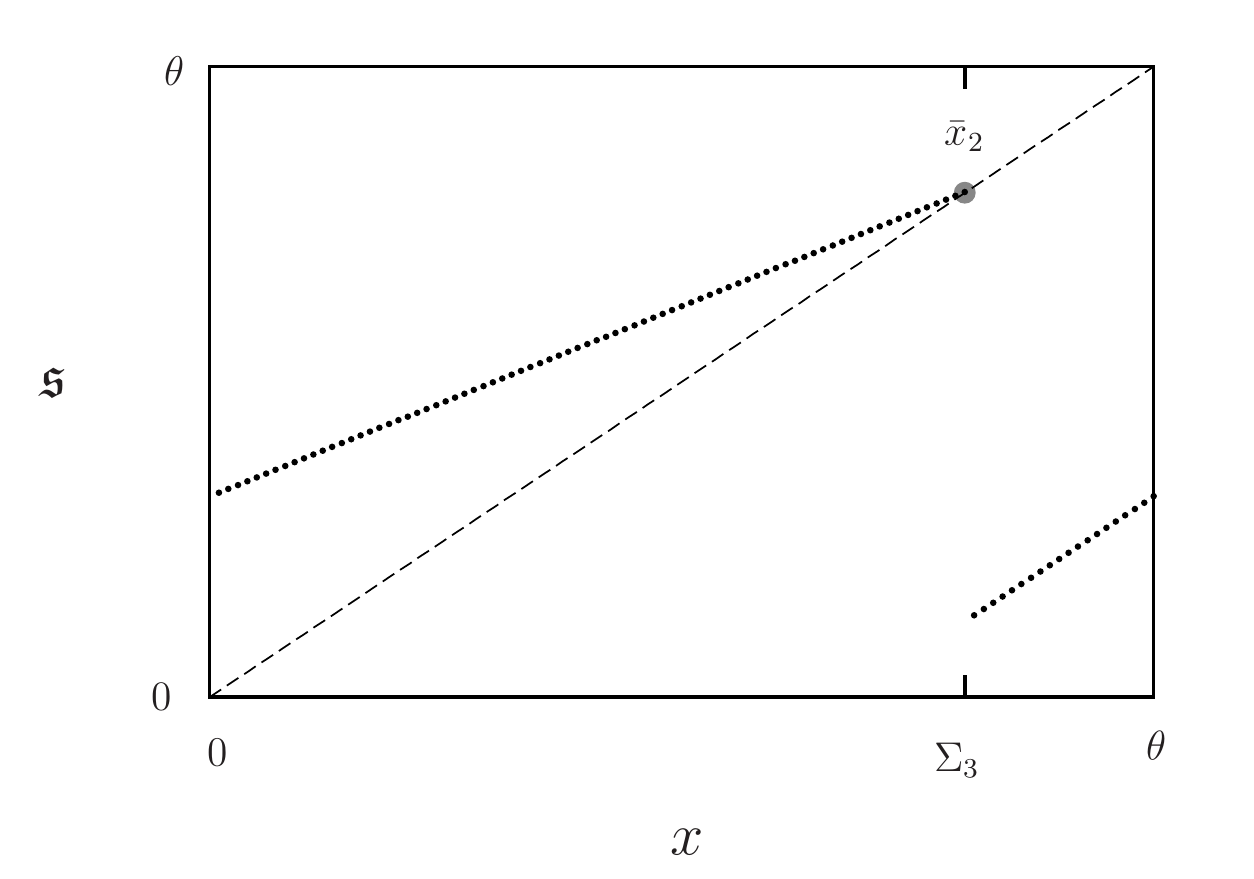}}
}
\end{picture}
\caption{Stroboscopic map for the $T$-periodic orbits shown in
Figure~\ref{fig:bif_T-per}. In (a) and (d) the fixed point $\bx_2$ undergoes
border collision bifurcation when it collides with the boundaries $\Sigma_2$
from the right and $\Sigma_3$ from the left, respectively. Note that, in (d) the
fixed point is shown in gray to emphasize that the map takes indeed the value on
the right for $x=\Sigma_3$.  In (b)-(c) the boundary $\Sigma_2$ disappears and a
new boundary $\Sigma_3$ appears while the fixed point $\bx_2$ remains.}
\label{fig:maps_bx2}
\end{figure}
The first bifurcation defining the uppermost bifurcation curve, given by
$\bx_0=\Sigma_1$ ($A=A_0(d)$), is a bit different than the others, as it
separates the parameter space $d\times1/A$ in two regions. In the lower side of
the bifurcation curve, only spiking asymptotic dynamics are possible whereas on
the upper side only a continuous $T$-periodic orbit exhibiting no spike can
exist.

When entering the white regions, the map does no longer possess any fixed point.
Instead, periodic orbits with arbitrarily high periods exist. These are shown in
Figure~\ref{fig:d-invA_generic_1dscann} along the segment shown in
Figure~\ref{fig:d-invA_generic_regions}. As can be observed, they are organized
by the period adding structure;  that is, between two periodic orbits with
periods $p$ and $q$, there exists another periodic orbit with period $p+q$.

As usual in piecewise-smooth dynamics, one can encode periodic orbits by
introducing symbolic dynamics as follows. We assign the letters $\LL$ and $\R$
depending on whether the corresponding periodic orbit steps on the left or on
the right of the discontinuity. Then, the adding phenomenon is given by the
concatenation of symbolic sequences; that is, between two regions in parameter
space where the periodic orbits with symbolic sequences $\sigma$ and $\omega$
exist, there exists a region locating a periodic orbit with symbolic sequence
$\sigma\omega$, whose period is the addition of the two previous ones.  In
Figure~\ref{fig:farey_tree} we show the symbolic sequences of the periodic
orbits found along the line shown in Figure~\ref{fig:d-invA_generic_regions}
when crossing the white region between points $B$ and $C$, as well as their
associated rotation numbers. These numbers are obtained by dividing the number
of $\R$'s contained in the symbolic sequence by its total length (the period of
the periodic orbit).\\
Note that the rotation numbers are organized by the so-called Farey
tree associated with the period-adding phenomenon. Other
authors~\cite{FreGal11B} suggest that this should be given by a
Stern-Brocot tree. However, in the context of the period adding, it is
the Farey tree that needs to be considered, as it contains more
precise information in the form of rotation numbers (see for
example~\cite{GamGleTre84,GamLanTre84} and~\cite{AlsGamGraKru14} for a
recent survey.). As a consequence, the rotation number follows a
devil's staircase from $0$ to $1$. This is a monotonically increasing
function which is constant almost everywhere, except in a Cantor set
of zero measure.

\begin{figure}
\begin{center}
\begin{tikzpicture}[->,>=stealth',shorten >=1pt,auto,node distance=1cm]
\node[] (a1){};
\node[] (a2) [right of=a1] {};
\node[] (a3) [right of=a2] {};
\node[] (a4) [right of=a3] {$\LL/0$};
\node[] (a5) [right of=a4] {};
\node[] (a6) [right of=a5] {$\R/1$};
\node[] (a7) [right of=a6] {};
\node[] (a8) [right of=a7] {};
\node[] (a9) [right of=a8] {};
\node[] (b1) [above of=a1] {};
\node[] (b2) [right of=b1] {};
\node[] (b3) [right of=b2] {};
\node[] (b4) [right of=b3] {};
\node[] (b5) [right of=b4] {$\LL\R/\frac{1}{2}$};
\node[] (b6) [right of=b5] {};
\node[] (b7) [right of=b6] {};
\node[] (b8) [right of=b7] {};
\node[] (b9) [right of=b8] {};
\node[] (c1) [above of=b1] {};
\node[] (c2) [right of=c1] {};
\node[] (c3) [right of=c2] {$\LL^2\R/\frac{1}{3}$};
\node[] (c4) [right of=c3] {};
\node[] (c5) [right of=c4] {};
\node[] (c6) [right of=c5] {};
\node[] (c7) [right of=c6] {$\LL\R^2/\frac{2}{3}$};
\node[] (c8) [right of=c7] {};
\node[] (c9) [right of=c8] {};
\node[] (d1) [above of=c1] {};
\node[] (d2) [right of=d1] {$\LL^3\R/\frac{1}{4}$};
\node[] (d3) [right of=d2] {};
\node[] (d4) [right of=d3] {};
\node[] (d5) [right of=d4] {};
\node[] (d6) [right of=d5] {};
\node[] (d7) [right of=d6] {};
\node[] (d8) [right of=d7] {$\LL\R^3/\frac{3}{4}$};
\node[] (d9) [right of=d8] {};
\node[] (e1) [above of=d1] {$\LL^4\R/\frac{1}{5}$};
\node[] (e2) [right of=e1] {};
\node[] (e3) [right of=e2] {};
\node[] (e4) [right of=e3] {$\LL^2\R\LL\R/\frac{2}{5}$};
\node[] (e5) [right of=e4] {};
\node[] (e6) [right of=e5] {$\LL\R\LL\R^2/\frac{3}{5}$};
\node[] (e7) [right of=e6] {};
\node[] (e8) [right of=e7] {};
\node[] (e9) [right of=e8] {$\LL\R^4/\frac{4}{5}$};
\path[]
(a4) edge node [left] {} (b5)
(a6) edge node [left] {} (b5)
(b5) edge node [left] {} (c3)
(a4) edge node [left] {} (c3)
(b5) edge node [left] {} (c7)
(a6) edge node [] {} (c7)
(c7) edge node [] {} (d8)
(a6) edge node [] {} (d8)
(a4) edge node [left] {} (d2)
(c3) edge node [left] {} (d2)
(a4) edge node [left] {} (e1)
(d2) edge node [left] {} (e1)
(a6) edge node [] {} (e9)
(d8) edge node [] {} (e9)
(b5) edge node [] {} (e4)
(c3) edge node [] {} (e4)
(b5) edge node [] {} (e6)
(c7) edge node [] {} (e6)
;
\end{tikzpicture}
\end{center}
\caption{Symbolic sequences and rotation numbers forming the so-called Farey
tree structure for the period adding. The symbol $\LL$ corresponds to a step on
the left of the discontinuity, $\R$ to a step on its right.}
\label{fig:farey_tree}
\end{figure}
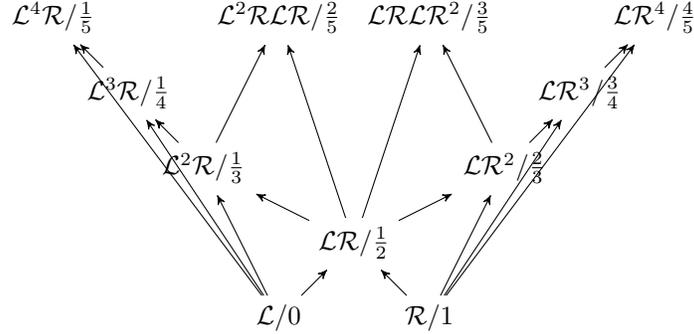
Immediately after crossing a white region and entering a gray one where another
$T$-periodic orbit exists (fixed point of the stroboscopic map), the rotation
number equals $1$ for a while until it suddenly jumps to $0$ again. This is due
to the following reason.\\
When varying parameters along the line shown in~\ref{fig:d-invA_generic_regions}
inside the gray regions, a new discontinuity $\Sigma_{n+1}$ enters $[0,\theta]$
($\theta=\Sigma_{n+1}$) and $\Sigma_n$ no longer exists (because of the
uniqueness of the discontinuity mentioned above), while the periodic orbit
spiking $n$ times still exists and hence is not subject to any bifurcation (see
Figures~\ref{fig:map2}-\ref{fig:map3} for $n=2$).  At this moment, however, the
rotation number associated with this periodic orbit jumps from $1$ to $0$, and
the state space is now split in two pieces, $[0,\Sigma_{n+1})$ and
$[\Sigma_{n+1},\theta)$ where the system spikes $n$ and $n+1$ times,
respectively. This comes from the fact that, when a new discontinuity
appears, what was on the right of the previous discontinuity $\Sigma_n$
becomes on the left of the new one $\Sigma_{n+1}$; hence, the stroboscopic map
can be reduced to a new circle map which rotates on the opposite direction.

\begin{remark}\label{rem:symbolic_seq-spikes}
The symbols $\LL$ and $\R$ in the symbolic sequences of the periodic orbits
located in the white regions correspond to $n$ and $n+1$ spikes in a $T$-time
interval, respectively.
\end{remark}
\begin{remark}\label{rem:rotation_numbers-eta_B-C}
Following~\cite{GamLanTre84}, one can relate the rotation numbers to the
symbolic dynamics associated with the periodic orbits that appear along the line
shown in Figure~\ref{fig:d-invA_generic_regions}, by dividing the number of
$\R$'s that appear in the symbolic sequence by the period of the periodic orbit.
Hence, taking into account Remark~\ref{rem:symbolic_seq-spikes}, between $B$ and
$C$ in Figure~\ref{fig:d-invA_generic_regions} the rotation number equals the
firing-number $\eta$.
\end{remark}

\begin{remark}\label{rem:rotation_numbers-eta}
Beyond point $C$ in the line shown in fig.~\ref{fig:d-invA_generic_regions},
$\eta$ varies along the line as the rotation number but without the jumping from
$1$ to $0$. Hence, this quantity follows a devil's staircase from $0$ to
$\infty$ when parameters are varied along such a line.
\end{remark}

\begin{remark}\label{rem:contractiveness}
In the conditions mentioned at the beginning of this section, in
addition to~\conds{} it was also required that $T$ be large or small
enough. This is needed in order to ensure that the stroboscopic map
$\s$ is contractive in all its domain, which is a necessary condition
for the occurrence of the period adding. It is possible, for certain
values of $T$, that, when $A$ is not sufficiently large, the
stroboscopic map $\s$ be expanding in the domain $[\Sigma_n,\theta]$.
When this occurs, the rotation number may no longer follow a devil's
staircase but a continuous increasing function, the existence of a
periodic orbit may not be unique and it can unstable.
See~\Firstpaper{} for more details.
\end{remark}

\subsection{Bifurcation scenario upon frequency variation}\label{sec:freq_properties}
We now focus on how the bifurcation scenario described
in~\S\ref{sec:bif_scenario} and schematically shown in
Figure~\ref{fig:d-invA_generic} varies with $T$.\\
As proven in~\Firstpaper{} this bifurcation scenario does not
qualitatively depend on $T$ and, hence, no other bifurcations are
introduced nor subtracted under variation of $T$ as long as
contractivness of the stroboscopic map $\s$ is kept (see
Remark~\ref{rem:contractiveness}). However, the shape of the
bifurcation curves varies, as the next two propositions show (see
Figures~\ref{fig:diff_Ts} and~\ref{fig:diff_largeTs} for graphical
support through an example).  Proposition~\ref{prop:Tlarge} tells us
that the bifurcation curves accumulate to the horizontal line
$1/A=1/Q_c$ when $T\to\infty$ (labeled in all paths of
Figures~\ref{fig:diff_Ts} and~\ref{fig:diff_largeTs}).
Proposition~\ref{prop:Tsmall} tells us that all bifurcation curves
accumulate to horizontal axis, except for the one given by $A_0(d)$,
which accumulates to the straight line $1/A=d/Q_c$ when $T\to0$ (see
Figure~\ref{fig:T0d1}).

\begin{proposition}\label{prop:Tlarge}
Let $d\in(0,1)$ and consider the values given in definition~\ref{def:AnR_AnL},
$A_n^{\R,\LL}=A_n^{\R,\LL}(d,T)$ and $A_0=A_0(d,T)$, for which the fixed points
$\bx_n$ undergo border collision bifurcations. Then,
\begin{equation*}
\lim_{T\to\infty}A_0(d,T)=\lim_{T\to\infty}A_n^\R(d,T)=\lim_{T\to\infty}A_n^\LL(d,T)=Q_c,
\end{equation*}
where $Q_c$ is the critical dose defined in~\eqref{eq:critical_dose_eq}.
\end{proposition}

\begin{proof}
Let $\varphi(t;x;A)$ be the flow associated with system $\dot{x}=f(x)+A$, and
let $\bx_n=\bx_n(d,T,A)$ be the initial condition ($t_0=0$) for a $T$-periodic
orbit spiking $n$ times (fixed point of the stroboscopic map $\s(x)$). As shown
in Figure~\ref{fig:bif1}, (see~\Firstpaper{} for more details), the border
collision bifurcations that the $T$-periodic orbit spiking $n$ times undergo at
$A=A_n^\R(d,T)$ (fig.~\ref{fig:bif1_L}) and $A=A_n^\LL$ (fig.~\ref{fig:bif1_R})
are characterized by the equations
\begin{align}
\varphi(t_n;\bx_n;A_n^\R)&=\theta\nonumber\\
\varphi(\delta;0;A_n^\R)&=\theta\label{eq:bif1}\\
\varphi(T-dT;0;0)&=\bx_n\nonumber\\
t_n+(n-1)\delta&=dT\nonumber
\end{align}
and
\begin{align}
\varphi(t_n';\bx_n;A_n^\LL)&=\theta\nonumber\\
\varphi(\delta';0;A_n^\LL)&=\theta\label{eq:bif2}\\
\varphi(T-dT;\theta;0)&=\bx_n\nonumber\\
t_n'+(n-1)\delta'&=dT\nonumber,
\end{align}
respectively. As $n$ is fixed, when $T\to\infty$ also $\delta\to\infty$. Hence,
from equations~\eqref{eq:bif1} and~\eqref{eq:bif2} we get that the values
$A=A_n^\R$ and $A=A_n^\LL(d,T)$ are such that system $\dot{x}=f(x)+A$ possesses an
attracting critical point at $x=\theta$. From
equation~\eqref{eq:critical_dose_eq} we get that the limiting values are
$A_n^\R=A_n^\LL=Q_c$.\\
Arguing similarly and using that the bifurcation condition for $A_0(d,T)$ is
equivalent to~\eqref{eq:bif2}, one gets that $A_0\to Q_c$.
\end{proof}
\begin{figure}
\begin{center}
\subfigure[\label{fig:bif1_L}]{
\includegraphics[width=0.5\textwidth]{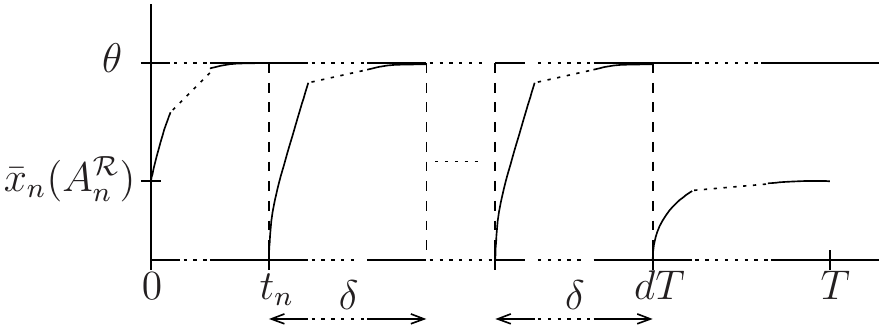}
}
\subfigure[\label{fig:bif1_R}]{
\includegraphics[width=0.5\textwidth]{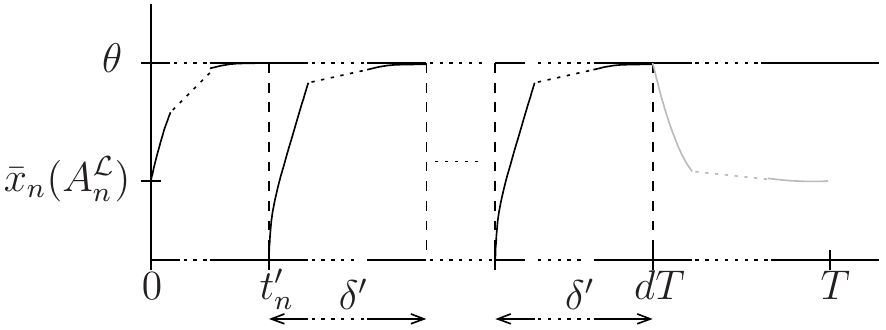}
}
\end{center}
\caption{Bifurcations of the $T$-periodic spiking $n$ times for large values of
$T$. (a) and (b) smallest and largest values of $A$ ($A_n^\R$ and $A_n^\LL$,
respectively) for which the $T$-periodic orbit exists. This corresponds to the
border collision bifurcations given by $\bx_n(A_n^\R)=\Sigma_{n}$ and $\bx_n(
(A_n^\LL)^-)=\Sigma_{n+1}$, respectively. The gray color in (b) reflects the
fact that such orbit does no longer exist for $A=A_n^\LL$, as it must be reset
to $0$ at $t=dT$. The orbit shown is the limiting periodic orbit for $A\to
A_n^\LL$.}
\label{fig:bif1}
\end{figure}

\begin{proposition}\label{prop:Tsmall}
Let $d\in(0,1)$ and consider the bifurcation values of the fixed points
$\bx_n\in S_n$, $A_0(d,T)$ and $A_n^{\R,\LL}(d,T)$ ($n>1$) as in
definition~\ref{def:AnR_AnL}. Then they fulfill
\begin{enumerate}[i)]
\item
\begin{equation}
\lim_{T\to0}A_0(d,T)=\frac{Q_c}{d},
\label{eq:limiting_A}
\end{equation}
where $Q_c$ is the critical dose defined in~\eqref{eq:critical_dose_eq}.
\item 
\begin{equation*}
\lim_{T\to0}A_n^{\R,\LL}(d,T)=\infty,\,n>1.
\end{equation*}
\end{enumerate}
\end{proposition}
\begin{proof}
Let $\varphi(t;x;A)$ be the flow associated with the system
$\dot{x}=f(x)+A$.  The fact that $f(x)$ is a monotonously decreasing
function with a simple zero (the equilibrium point $\bx$ in
\emph{H.}1) ensures that the bifurcation suffered by the non-spiking
$T$-periodic orbit will be given when
\begin{equation}
\left\{
\begin{aligned}
&\varphi(dT;\bx_0;A_0)=\theta\\
&\varphi(T-dT;\theta;0)=\bx_0,
\end{aligned}
\right.
\label{eq:bif_equation}
\end{equation}
where $\bx_0$ is the initial condition for such periodic orbit for $t=0$. We
want to solve these equations for $A_0$ and $\bx_0$ for a fixed $d\in(0,1)$ and
$T>0$, and see how this solution behaves when $T\to0$.\\
For $T\to0$ we can approximate the flow by a linear one and obtain
\begin{align*}
&\varphi(dT;x;A_0)=x+\left( f(x)+A_0 \right)dT+O(T^2)\\
&\varphi(T-dT;\theta;0)=\theta+f(\theta)(1-d)T+O(T^2).
\end{align*}
Hence, for $T\to0$, the bifurcation condition~\eqref{eq:bif_equation} becomes
equivalent to the system
\begin{equation*}
\left\{
\begin{aligned}
\bx_0+\left( f(\bx_0)+A_0 \right)dT=\theta\\
\theta+f(\theta)\left( 1-d \right)T=\bx_0,
\end{aligned}
\right.
\end{equation*}
which we can use to obtain the explicit expression
\begin{equation*}
A_0(d,T)=-\frac{f(\theta)}{d}+f(\theta)-f\big(\theta+f\left(\theta\right)(1-d)T\big)+O(T^2).
\end{equation*}
Recalling the definition of the critical dose given in
Eq.~\eqref{eq:critical_dose_eq}, $Q_c=-f(\theta)$, we obtain
\begin{equation*}
\lim_{T\to0}A_0(d,T)=\frac{Q_c}{d},
\end{equation*}
which proves {\it i)}.\\
To see that all other bifurcation curves accumulate at the horizontal
line $1/A=0$, as stated in {\it ii)}, we just use the fact that the
periodic orbits involved in these bifurcations perform at least one
spike for $t\in[0,dT]$.  Hence, when $T\to0$, we necessary have that
$A\to\infty$ in order to keep these spikes.
\end{proof}

\begin{remark}
Note that the fact that the border collision bifurcation curves defined by
$A_n^{\R,\LL}$ collapse to the horizontal axis for $T\to0$ implies that all
other border collision bifurcation curves separating regions of existence of
periodic orbits with higher periods also collapse to the horizontal axis, as
they are located in between.
\end{remark}

The next result tells us that all bifurcation curves vary monotonically with
$T$.
\begin{lemma}\label{lem:monotonicity}
For a fixed $d\in(0,1)$ let $A^\R(d,T)$ and $A^\LL(d,T)$ be the values for which
a periodic orbit undergoes a border collision bifurcation. Then they are
monotonic functions of $T$.
\end{lemma}

\begin{proof}
We prove the result for the bifurcations undergone by fixed points of the
stroboscopic map ($A^\LL_n(d,T)$ and $A_n^\R(d,T)$). Proceeding similarly one
obtains the analogous result for periodic orbits.

Assume that $\bx_n$ is a fixed point of the stroboscopic map $\s(x)$ leading to
a $T$-periodic orbit exhibiting $n$ spikes per period. Such a fixed point
undergoes a left bifurcation for $A=A_n^\LL$ when the following equations are
satisfied (see fig.~\ref{fig:T-per_L-bif})
\begin{align*}
&\int_{\bx_n}^\theta\frac{dx}{f(x)+A_n^\LL}+n\int_0^\theta\frac{dx}{f(x)+A_n^\LL}=aT\\
&\int_\theta^{\bx_n}\frac{dx}{f(x)}=(1-a)T,
\end{align*}
where we have renamed the duty cycle $d$ by $a$ to avoid the confusion with the
notation used for derivatives and differentials. Differentiating the previous
equations  with respect to $T$ we get
\begin{align*}
a&=-\frac{dA_n^\LL}{dT}\left(
\int_{\bx_n}^\theta\frac{dx}{\left(f(x)+A_n^\LL\right)^2}+n\int_0^\theta\frac{dx}{\left(
f(x)+A_n^\LL
\right)^2} \right)-\frac{1}{f(\bx_n)+A_n^\LL}\frac{d\bx_n}{dT}\\
\frac{d\bx_n}{dT}&=(1-a)f(\bx_n).
\end{align*}
We want to see that $dA_n^\LL/dT<0$. Combining these last equations we get
\begin{equation*}
-\left(\int_{\bx_n}^\theta\frac{dx}{\left(f(x)+A_n^\LL\right)^2}+n\int_0^\theta\frac{dx}{\left(
f(x)+A_n^\LL
\right)^2} \right)\frac{dA_n^\LL}{dT}=a+(1-a)\frac{f(\bx_n)}{f(\bx_n)+A_n^\LL}.
\end{equation*}
We note that
\begin{itemize}
\item the coefficient of $dA_n^\LL/dT$ in the left hand side is negative
\item $f(\bx_n)\left( f(\bx_n)+A_n^\LL \right)<0$.
\end{itemize}
Hence, to prove the result we need to show that
\begin{equation}\label{eq-toprove}
a+(1-a)\frac{f(\bx_n)}{f(\bx_n)+A_n^\LL}>0.
\end{equation}
We know that if $T\to\infty$ then $f(\bx_n)\to 0$ and $A_n^\LL\to Q_c$
(Proposition~\ref{prop:Tlarge}). Hence \eqref{eq-toprove} holds for large $T$.
We will prove that $a+(1-a)f(\bx_n)/(f(\bx_n)+A)=0$ is not possible.  First note
that this equation is equivalent to $\bx_n=f^{-1}(-aA)$.  Hence we consider the
following system of three equations:
\begin{equation}
\begin{aligned}
&\int_{\bx_n}^\theta \frac{dx}{f(x)+A_n^\LL}+n\int_0^\theta\frac{dx}{f(x)+A_n^\LL}=aT\\
&\int_\theta^{\bx_n}\frac{dx}{f(x)}=(1-a)T\\
&\bx_n=f^{-1}(-aA_n^\LL).
\end{aligned}
\label{eq-3eq}
\end{equation}
and we show that the three equations in~\eqref{eq-3eq} cannot be
simultaneously satisfied. Eliminating the variables we get
\begin{equation*}
\int_{f^{-1}(-aA_n^\LL)}^\theta\left ( \frac{a}{f(x)}+\frac{1-a}{f(x)+A_n^\LL}\right
)dx+ n\int_0^\theta\frac{(1-a)}{f(x)+A_n^\LL}dx=0.
\end{equation*}
On one hand, $A_n^\LL$ is large enough to make the system spike, and
hence $f(\theta)+A_n^\LL>0$. On the other hand, as $\bx_n<\theta$,
from the last equation of Eq.~\eqref{eq-3eq} we get that
$-aA_n^\LL<f(\theta)$.  Therefore, we know that $0<-f(\theta)\le
A_n^\LL< -f(\theta)/a$. We define the function
\begin{equation*} \psi(A)=\int_{f^{-1}(-aA)}^\theta\left (
\frac{a}{f(x)}+\frac{1-a}{f(x)+A}\right )dx+
n\int_0^\theta\frac{(1-a)}{f(x)+A}dx.  \end{equation*}
Clearly $\psi(-f(\theta)/a)>0$. We compute
\begin{align*}
\psi'(A)&=\frac{a}{f'(f^{-1}(-aA))}\left (-\frac{1}{A}+\frac{1}{A}\right )
-\int_{f^{-1}(-aA)}^\theta\frac{1-a}{(f(x)+A)^2} dx\\
&-n\int_0^\theta\frac{1-a}{\left(
f(x)+A
\right)^2}dx\\
&=  -\int_{f^{-1}(-aA)}^\theta\frac{1-a}{(f(x)+A)^2} dx -n\int_0^\theta\frac{1-a}{\left(
f(x)+A
\right)^2}dx< 0.
\end{align*}
Hence $\psi$ is strictly monotonic so it cannot have a $0$ in $[-f(\theta),
-f(\theta)/a)$. 

We now show that the same holds for the right bifurcations of the fixed points
(see fig.~\ref{fig:T-per_R-bif}); that is,
$dA_n^\R/dT<0$.\\
In this case, the equations that determine such bifurcation become
\begin{align*}
&\int_{\bx_n}^\theta\frac{dx}{f(x)+A_n^\R}+n\int_0^\theta\frac{dx}{f(x)+A_n^\R}=aT\\
&\int_0^{\bx_n}\frac{dx}{f(x)}=(1-a)T,
\end{align*}
Unlike in the previous case, $\bx_n$ increases to $\bx$ (the
equilibrium point assumed in \emph{H}.1) when $T$ is increased.  This
leads to a decrease of the time of the first spike (value of the first
integral). Hence, one necessary needs to decrease $A_n^\R$ in order to
keep these equations satisfied.
\end{proof}

We will now use Propositions~\ref{prop:Tlarge} and~\ref{prop:Tsmall} to derive
information about the behavior of the firing-rate for large and small periods,
Propositions~\ref{prop:fr_Tlarge} and~\ref{prop:fr_Tsmall}, respectively. First
we present the next corollary of
Propositions~\ref{prop:Tlarge},~\ref{prop:Tsmall} and
Lemma~\ref{lem:monotonicity}. It provides a partition of the parameter space in
three different regions regarding spiking properties for different values of
$T$.
\begin{corol}\label{cor:regions}
The parameter space $d\times 1/A$ is divided in three main regions with the
following properties
\begin{itemize}
\item \emph{Non-spiking region},
\begin{equation}
\left\{ (d,1/A)\in\RR^2\;|\;d\in(0,1),\,A<Q_c \right\},
\label{eq:subthreshold_region}
\end{equation}
for which the corresponding periodic orbit does not contain any spike for any
$d\in(0,1)$ and $T>0$.
\item \emph{Permanent-spiking region},
\begin{equation}
\left\{ (d,1/A)\in\RR^2\;|\;d\in(0,1),\;A>Q_c/d \right\},
\label{eq:sbcrictical_region}
\end{equation}
for which, the existing periodic orbit contains spikes for all
$T>0$.
\item \emph{Conditional-spiking region},
\begin{equation}
\left\{ (d,1/A)\in\RR^2\;|\;d\in(0,1),\,Q_c<A<Q_c/d \right\},
\label{eq:supercritical_region}
\end{equation}
for which there exists $T_0>0$ such that the corresponding periodic orbit
contains spikes if $T>T_0$ and does not if $T<T_0$.
\end{itemize}
\end{corol}

\begin{remark}\label{rem:spiking_region}
The spiking-region is formed by the union of the conditional and
permanent-spiking regions.
\end{remark}

\begin{corol}\label{coro:devils_staircase}
If $(d,1/A)$ belongs to the spiking-region, then, for those values of
$T$ for which $\s$ is contracting in $[0,\theta]$, the firing-rate
$r(T)$ follows a devil's staircase with monotonically decreasing
steps. For the values of $T$ for which $\s$ loses contractiveness
$r(T)$ is a monotonically increasing function.
\end{corol}
\begin{proof}
It follows from Propositions~\ref{prop:Tlarge},~\ref{prop:Tsmall} and
Lemma~\ref{lem:monotonicity} that, as long as $T$ is such that $\s$ is
contracting (see Remark~\ref{rem:contractiveness}), then the
bifurcation lines defining the steps of the devil's staircase move up
monotonically as $T$ is increased.  Hence, if we fix $(d, 1/A)$ in the
spiking region, then as $T$ increases all the bifurcation curves pass
through the point $(d, 1/A)$ and the bifurcation diagram is the same
as when $d$ and $T$ are fixed and $A$ is varied (i.e. when varying
parameters along the line shown in
fig.~\ref{fig:d-invA_generic_regions} for a fixed $T$).\\
Recall that the the rotation number follows a devil's staircase which
is constant along the steps.  Then, using
Remarks~\ref{rem:rotation_numbers-eta_B-C}
and~\ref{rem:rotation_numbers-eta} and  formula~\eqref{eq:firing_rate}
we conclude that $r(T)$ follows a devil's staircase which is
monotonically decreasing along the steps.

If eventually $\s$ becomes expanding in $[\Sigma_n,\theta]$, then the
rotation number does not follow a devil's staircase but a
monotonically increasing continuous function
(see~\cite{GraKruCle13,AlsGamGraKru14} for more details).
\end{proof}

From the previous results we get the next corollary providing the behaviour of
the firing number $\eta$ for large and small values of $T$.
\begin{corol}\label{cor:eta_smallT}
In the spiking region the firing number defined
in~\ref{def:firing_number} (average number of spikes per iteration of
the stroboscopic map) satisfies
\begin{align*}
\lim_{T\to 0} \eta(T)&=0\\
\lim_{T\to\infty}\eta(T)&=\infty.
\end{align*}
Morover, in the conditional spiking region $\eta(T)=0$ for $0<T<T_0=T_0(A,d)$.
\end{corol}
Note that the relation between the firing number and firing-rate given by
equation~\eqref{eq:firing_rate} implies that their asymptotic behavior is not
necessarily the same when $T\to\infty$ or $T\to 0$.  We will now use the results
obtained so far in this section to characterize the limits of the firing-rate as
$T\to\infty$ and as $T\to 0$.  The following result describes the limit of
$r(T)$ as $T\to\infty$ for points in the spiking region.
\begin{proposition}\label{prop:fr_Tlarge}
Let $(d,1/A)$ belong to the spiking-region ($d\in(0,1)$ and $A>Q_c$)
and let $\varphi(t;x;A)$ be the flow associated with $\dot{x}=f(x)+A$.
Let $\delta>0$ be the smallest value such that
\begin{equation*}
\varphi(\delta;0;A)=\theta.
\end{equation*}
Then, the firing-rate satisfies
\begin{equation*}
\lim_{T\to\infty}r(T)=\frac{d}{\delta}.
\end{equation*}
\end{proposition}

\begin{proof}
We will show that the devil's staircase followed by the firing number
$\eta(T)$ converges to a common discontinuous staircase whose each
step has length $\delta/d$. These steps have integer values, $n$, and
correspond to $T$-periodic orbits spiking $n$ times. More precisely,
we prove that
\begin{equation}\label{eq:MKlim0}
\lim_{T\to\infty}\eta(T)-E\left( \frac{\delta T}{d} \right)\to0,
\end{equation}
where $E(x)$ is the integer value of $x$.  The result follows from
\eqref{eq:MKlim0} and from the definition of firing-rate.\\

\noindent To prove \eqref{eq:MKlim0} we focus on a $T$-periodic orbit with $n$
spikes per period. Let $\bx_n(T)$ be the initial condition ($t_0=0$) for such an
orbit (fixed point of the stroboscopic map). This fulfills
\begin{align*}
\varphi(t_n;\bx_n(T);A)&=\theta\\
\varphi(dT-(n-1)\delta-t_1;0;A)&=x'\\
\varphi(T-dT;x';0)&=\bx_n(T),
\end{align*}
for some $0<t_n<\delta$. The last equation tells us that
\begin{equation}
\lim_{T\to\infty}\bx_n(T)=\bx,
\label{eq:xn_infty}
\end{equation}
where $\bx$ is the equilibrium point~\eqref{eq:critical_point} associated with
system $\dot{x}=f(x)$ given by assumptions~\conds{}.\\
At the same time, this tells us that the stroboscopic map converges to a
constant function equal to $\bx$. Recalling that the discontinuities of the
stroboscopic map occur at $x=\Sigma_i$, the gaps at these discontinuities tend
to zero,
\begin{equation*}
\lim_{T\to\infty}\s(\Sigma_n^-)=\lim_{T\to\infty}\s(\Sigma_n^+)=\bx.
\end{equation*}
Hence, when $T\to\infty$ there is no space for periodic orbits with higher
periods. In other words, let $T_n^\R$ and $T_n^\LL$ be the values of $T$ for
which a $T$-periodic orbit spiking $n$ times appears and disappears through
border collisions bifurcations, respectively (Figures~\ref{fig:bif_R}
and~\ref{fig:bif_L}, respectively). Then we have that
\begin{equation*}
\lim_{T\to\infty}T_{n+1}^\R-T_n^\LL=0,
\end{equation*}
and the devil's staircase converges to be a common staircase. Its steps are given by
integer values $n$, as they correspond to the firing-number associated with
$T$-periodic orbits spiking $n$ times.

\noindent We now estimate the length of these steps when $T\to\infty$.
From equation~\eqref{eq:xn_infty} we get that, as $T\to \infty$, $t_1$ converges
to the solution of the equation 
\begin{equation}\label{eq:MKlimt1}
\varphi(t_1;\bx;A)=\theta.
\end{equation}
As, for a fixed value of $A$ and $d$ in the spiking region, the number
of spikes performed by a $T$-periodic orbit tends to infinity as
$T\to\infty$ (Proposition~\ref{prop:Tlarge}), the interval of time
where the spikes occur is of order $n\delta$. Hence, taking into
account the characteristics of the $T$-periodic orbit at its
bifurcation (see Figures~\ref{fig:bif_L}, \ref{fig:bif_R}
and~\Firstpaper), we get
\begin{align*}
dT_n^\R&\sim n\delta\\
dT_n^\LL&\sim (n+1)\delta,
\end{align*}
and thus
\begin{equation*}
T_n^\LL-T_n^\R\sim \frac{\delta}{d},
\end{equation*}
which is the length of the step with integer value.
\begin{figure}
\begin{center}
\subfigure[\label{fig:bif_R}]{
\includegraphics[width=0.5\textwidth]{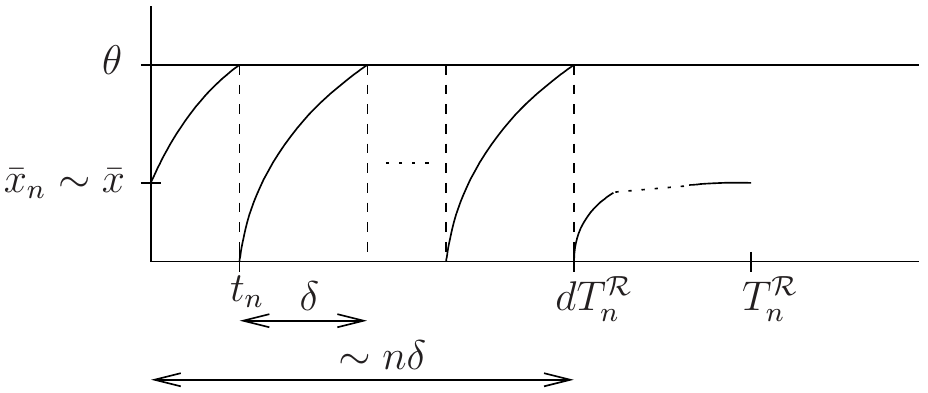}
}
\subfigure[]{
\includegraphics[width=0.5\textwidth]{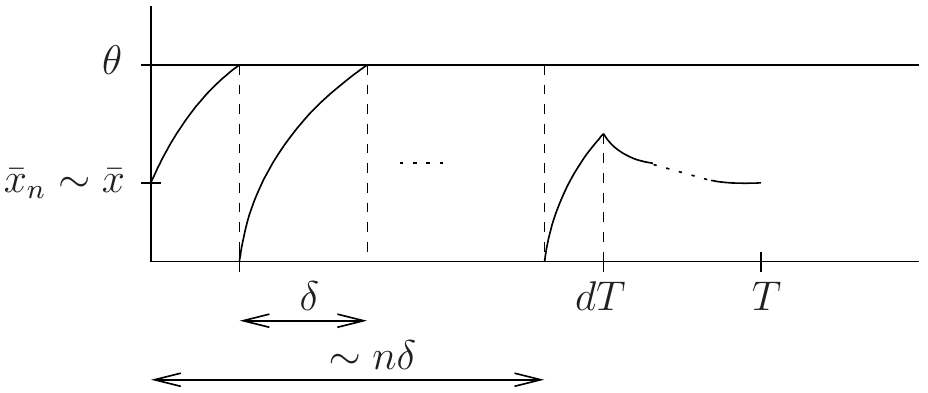}
}
\subfigure[\label{fig:bif_L}]{
\includegraphics[width=0.5\textwidth]{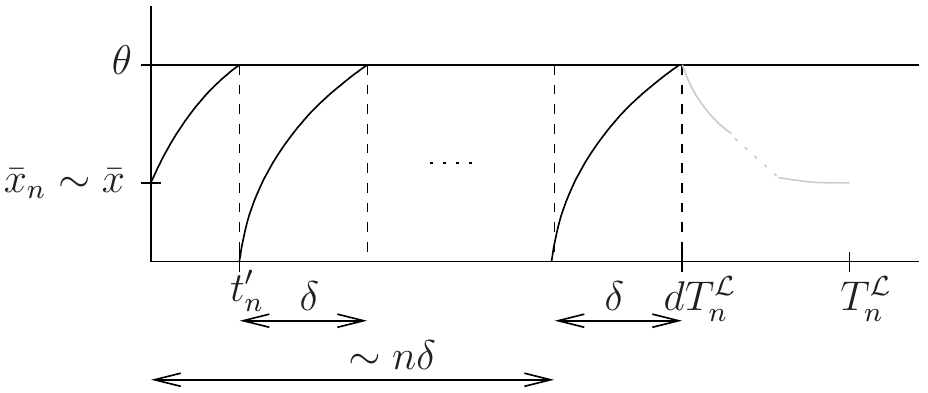}
}
\end{center}
\caption{$T$-periodic orbit spiking $n$ times at its bifurcations for large
values of $T$.  The periodic orbit appears (a) and disappears (c) through border
collision bifurcations at $T=T_n^\R$ and $T=T_n^\LL$, respectively. In (b) a
$T$-periodic orbit for $T\in(T_n^\R,T_n^\LL)$. The gray color emphasizes the
fact that the periodic orbit does not exist for $T=T_n^\LL$ as it should be
reset to $0$ when the threshold is reached. The orbit shown is the limiting
periodic orbit when $T\to \left(T_n^\LL\right)^-$. In all three cases the fixed
point $\bx_n$ approaches $\bx$, the critical point of
system~\eqref{eq:autonomous_system} as $T\to\infty$.}
\label{fig:T-periodic_bifs}
\end{figure}
\end{proof}
We end this section with a result which describes the behavior of $r(T)$
as $T\to 0$.
\begin{proposition}\label{prop:fr_Tsmall}
Let $(d,1/A)$ belong to the spiking-region,
and let
\begin{equation}
\dot{x}=f(x)+Ad
\label{eq:averaged_system}
\end{equation}
be the averaged version of system~\eqref{eq:general_system}. Let
$\hphi(t;x)$ be its associated flow and let $\hdelta>0$ be the
smallest number such that
\begin{equation}
\hphi(\hdelta;0)=\theta.
\label{eq:averaged_collision}
\end{equation}
Then,
\begin{itemize}
\item if $(d,1/A)$ belongs to the conditional-spiking region ($Ad<Q_c$) then
$r(T)=0$ if $T<T_0$, where $T_0$ is given in Corollary~\ref{cor:regions}, 
\item if $(d,1/A)$ belongs to the permanent-spiking region ($Ad>Q_c$), then
\begin{equation*}
\lim_{T\to0}r(T)=\frac{1}{\hdelta}.
\end{equation*}
\end{itemize}
\end{proposition}
\begin{proof}
For the first case we use Corollary~\ref{cor:regions}, from which we get that,
if $T<T_0$ then $\eta(T)=0$, and hence $r(T)=0$.

For the second case we study how the devil's staircase $r(T)$ behaves when
$T\to0$. Note that, when $r(T)<1$, this one coincides with the rotation number
of the periodic orbits found when varying $T$ (see
Remark~\ref{rem:rotation_numbers-eta_B-C}).\\
Using that $\lim_{T\to0}\eta(T)=0$, we get that, for any $T$ small enough, we
can find $n$ large enough such that
\begin{equation*}
\frac{1}{n+1}\le\eta(T)\le\frac{1}{n}.
\end{equation*}
Hence, as $1/n-1/(n+1)\to0$, it is enough to study how the steps given
by the rotation numbers of the form $1/n$ behave.  Taking into account
that the symbolic dynamics is organized by a Farey tree structure (a
one to one mapping with the rotation numbers), this rotation numbers
are associated with periodic orbits with symbolic sequences of the
form $\LL^n\R$. These periodic orbits are characterized by exhibiting
one spike after $n$ iterations of the stroboscopic map, and are
determined by the equations
\begin{align*}
\varphi(dT;\bx_{\LL^n\R};A)&=x_1\\
\varphi(T-dT;x_1;0)&=x_1'\\
\varphi(dT;x_1';A)&=x_2\\
\varphi(T-dT;x_2;0)&=x_2'\\
\vdots\\
\varphi(t';x_n';A)&=\theta\\
\varphi(dT-t';0;A)&=x_{n+1}\\
\varphi(T-dT;x_{n+1};0)&=\bx_{\LL^n\R},
\end{align*}
where $\varphi(t;x;A)$ is the flow associated with system $\dot{x}=f(x)+A$ and
$\bx_{\LL^n\R}$ is the initial condition for the $\LL^n\R$ periodic orbit for
$t_0=0$. From the two last equations, we get that $\lim_{T\to0}\bx_{\LL^n\R}=0$.\\
After applying a time rescaling, the original system~\eqref{eq:general_system} and
its averaged version become
\begin{align}
\dot{x}&=T\left( f(x)+\tilde{I}(t) \right)\label{eq:rescaled_system}\\
\dot{x}&=T\left( f(x)+Ad \right)\label{eq:rescaled_averaged},
\end{align}
where $\tilde{I}(t)$ is now $1$-periodic.
We now consider solutions of systems~\eqref{eq:rescaled_system}
and~\eqref{eq:rescaled_averaged} with $O(T)$ close initial conditions. The
averaging theorem of Bogoliubov and Mitropolski~\cite{BogMit61} tells us that,
if $T$ is small enough, then such solutions remain $O(T)$-close for a $t\sim
1/T$ time scale provided that they have not reached the threshold. Note that the
result given in~\cite{BogMit61} applies because it does not require continuity
in $t$ but boundedness and Lipschitz in $x$.\\
Hence, letting $\hphi(t;x)$ be the flow of the averaged
system~\eqref{eq:averaged_system}, if $T$ is small enough we have that
\begin{equation*}
\varphi(T-dT;\varphi(dT;x;A);0)=\hphi(T;x)+O(T).
\end{equation*}
Hence, as long as the threshold is not reached, we can approximate the real flow
by the averaged one. Using that $\bx_{\LL^n\R}\to0$ when $T\to0$, the time taken
by the real flow to reach the threshold $x=\theta$ from $x=\bx_{\LL^n\R}$
approaches $\hdelta$:
\begin{equation*}
nT+t'\to\hdelta.
\end{equation*}
Hence, as $t'<dT$, $n$ grows like $\hdelta/T$ when $T\to0$ and thus
\begin{equation*}
\lim_{T\to0}r(T)=\lim_{T\to0}\frac{1}{nT}=\frac{1}{\hdelta}.
\end{equation*}
\end{proof}
\begin{remark}\label{rem:r_depdends_on_Q}
Note that equation~\eqref{eq:averaged_collision} makes sense only if
the averaged system~\eqref{eq:averaged_system} has a stable critical
point above the threshold, which occurs if $Ad>Q_c$.  This occurs only
in the permanent-spiking region.
\end{remark}
\begin{remark}
From Proposition~\ref{prop:Tsmall} we get that the value of the firing-rate for
small frequencies depends on the released dose $Ad$, the average of
$I(t)$. However, for large values of $T$ (Proposition~\ref{prop:Tlarge}), it
depends explicitly on $A$ and $d$.
\end{remark}

\subsection{Optimization of the firing-rate}\label{sec:optimization}
As shown in Corollary~\ref{coro:devils_staircase}, the firing-rate as
a function of $T$, the period of the forcing $I(t)$, follows a devil's
staircase with monotonically decreasing pieces (see
Figure~\ref{fig:fr_typical}). This occurs for most values of $T$
except, possibly, in a bounded set, for which the firing-rate is an
increasing function. Each of the pieces forming the devil's
staircase occurs in a $T$-interval, $[T_\sigma^\R,T_\sigma^\LL]$, for
which a unique periodic orbit $\sigma$ exists. Hence, the firing-rate
exhibits local maxima at $T=T_\sigma^\R$, and local minima at
$T=T_\sigma^\LL$. As a consequence of this there exists an infinite
number of local minima and maxima at any interval of the form
$[T_\sigma^\LL,T_\gamma^\R]$, with $\sigma$ and $\gamma$ meaning
different periodic orbits.\\
Of particular interest is when $\sigma$ and $\gamma$ are consecutive
fixed points (spiking $n$ and $n+1$ times), because they occupy the
largest regions in parameter space and their rotation numbers bound
the ones of the periodic orbits, given by alternation of $n$ and $n+1$
spikes. Restricting to this case, we consider the firing-rate in the
frequency range corresponding to $[T_n^\LL,T_{n+1}^\LL]$. We will
prove that the firing-rate follows a devil's staircase with
monotonically decreasing steps but whose envelope is bell shaped; that
is, it increases from $T_n^\LL$ to $T_{n+1}^\R$, where it exhibits an
absolute maximum, and then decreases to $T_{n+1}^\LL$ (see
fig.~\ref{fig:fr_bell}). Note that the bifurcation values $T_n^\LL$,
$T_{n+1}^\R$ and $T_{n+1}^\LL$ can be easily found numerically (by
solving equations~\eqref{eq:bif1} and~\eqref{eq:bif2} for $T$) and
that the values of the firing-rate at these values become $n/T_n^\LL$,
$(n+1)/T_{n+1}^\R$ and $(n+1)/T_{n+1}^\LL$, respectively. In real
applications one is usually restricted to a bounded range of realistic
frequencies for which one observes an absolute maximum of the firing
rate (see for example~\cite{KaiJakSteChi97,dalkin_89}). Hence, this
approach could be applied to properly tune system parameters in order
to make the model exhibit such a behavior for the desired values of
$T$.
\begin{figure}
\begin{center}
\includegraphics[width=0.4\textwidth,angle=-90]{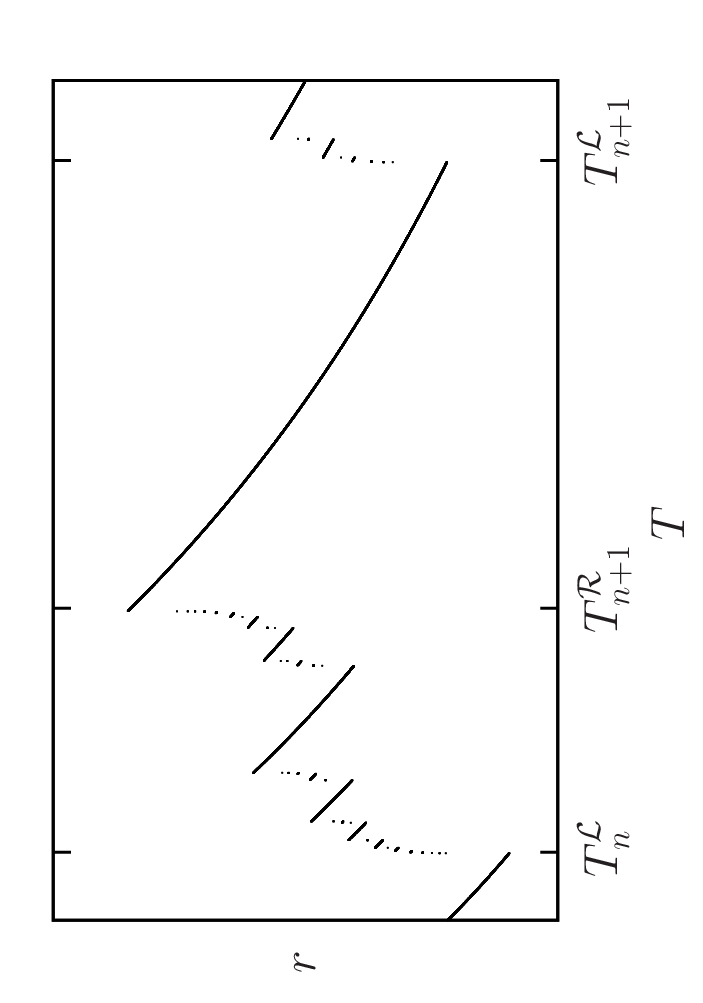}
\end{center}
\caption{Typical response of the firing-rate between two consecutive fixed
points. Its envelope is bell shaped, exhibiting a maximum at $T=T_{n+1}^\R$.}
\label{fig:fr_bell}
\end{figure}

We now investigate the optimization of the firing-rate in the whole range of
periods, $(0,\infty)$.\\
Due to the fact that the firing-rate is bounded and continuous for $T>0$, it
must exhibit a global maximum provided that it is increasing for $T=0^+$. From
the argument above, it must occur at some value of the form $T_\sigma^\R$,
respectively, for some periodic orbit $\sigma$. The next result tells us that,
in general, this periodic orbit will be the $T$-periodic orbit spiking once per
period.

\begin{proposition}\label{prop:optimization}
Let $(d,A)$ be in the spiking region (see Remark~\ref{rem:spiking_region}), and
let $T_1^\R$ and $T_1^\LL$ be the values of $T$ for which the periodic orbit
spiking once per period undergoes right and left border collision. Then, there
exists some $\gamma>0$ such that, if $T_1^\R>\gamma$ then the firing-rate $r(T)$
has global maximum at $T=T_1^\R$.
\end{proposition}

\begin{proof}
Let $T_n^{\R,\LL}$ be the values of $T$ for which a $T$-periodic orbit spiking
$n$ times undergoes border collision bifurcation on the right and left,
respectively. As we know, the firing-number is a monotonically increasing
function from $T_n^\LL$ to $T_{n+1}^\R$, for any $n$. Hence, the maximum must
occur for some $T_n^\R$, right border collision bifurcation of the $T$-periodic
orbit spiking $n$ times (see Figure~\ref{fig:T-per_R-bif} for $n=2$).\\ Taking
into account relation~\eqref{eq:firing_rate} and recalling that the firing
number for such orbits is $n$, the number of spikes, it will be enough to show
that
\begin{equation*}
\frac{1}{T_1^\R}>\frac{2}{T_2^\R}>\frac{3}{T_3^\R}\dots,
\end{equation*}
if $T_1^\R$ is large enough in order to see that this periodic orbit spikes once
per period.

Let $\varphi(t;x_0;A)$ and $\delta>0$ be as in Proposition~\ref{prop:fr_Tlarge},
and let $\bx_n$ be the fixed point of the stroboscopic map leading to the
$T$-periodic orbit spiking $n$ times. Then $T_n^\R$ is determined by the
following equations (see Figure~\ref{fig:T-per_R-bif} for $n=2$)
\begin{align*}
\varphi(t_n;\bx_n;A)&=\theta\\
\varphi(T_n^\R-dT_n^\R;0;0)&=\bx_n\\
t_n+(n-1)\delta&=dT_n^\R.
\end{align*}
As $T_{n+1}^\R>T_{n}^\R$, recalling that the flow $\varphi(t;x_0;0)$ is
exponentially attracted by the equilibrium point $\bx_n$, from the second
equation it comes that, at the moment of the bifurcation
\begin{align*}
\bx_n<\bx_{n+1}&\to\bx\\
t_n>t_{n+1}&\to \bar{t},
\end{align*}
where $\bar{t}$ is the smallest such that $\varphi(\bar{t};\bx;A)=\theta$.\\
As these series converge exponentially (due to the hyperbolicity of $\bx$), we
have that there exists some $N>0$, $0<\lambda<1$ and $K>0$ such that
\begin{equation}
n\underbrace{(t_n-t_{n+1})}_{<K\lambda^n}+t_n\to \bar{t},\,n>N.
\label{eq:tn_convergence}
\end{equation}
Assuming $n>N$ large enough and using that
$\delta>\bar{t}$ we get 
\begin{align}
\delta>n\left( t_n-t_{n+1} \right)+t_n&\Longleftrightarrow
nt_{n+1}+\delta>(n+1)t_n\label{eq:inequality}\\
&\Longleftrightarrow
n\overbrace{(t_{n+1}+n\delta)}^{T_{n+1}^\R}>(n+1)\overbrace{(t_n+(n-1)\delta)}^{T_n^\R}\\
&\Longleftrightarrow \frac{n}{T_n^\R}>\frac{n+1}{T_{n+1}^\R}.
\end{align}
In particular, if $T_1^\R$ is large enough, $\bx_1$ is close enough to $\bx$ to
fulfill~\eqref{eq:tn_convergence} and~\eqref{eq:inequality} for $n=1$.
\end{proof}

\begin{remark}
Note that, $T_1^\R$ will be large enough if $\bx$ is attracting enough.
\end{remark}

\begin{remark}\label{rem:global_minimum}
Arguing similarly, the global minimum will be the minimum of $0$ (if $(d,A)$
belongs to the conditional spiking region), $1/\hdelta$
(Proposition~\ref{prop:fr_Tsmall}) and $1/T_1^\LL$. Note that if the minimum
corresponds to $1/\hdelta$, then it technically does not exist, as $T=0$ is
excluded from the domain.
\end{remark}

\begin{corol}
For a given system~\SYSTEMWR{}, the globally maximal firing-rate is achieved
with the combination of period $T$, dose $Q$, amplitude $A$ and duty cycle $d$
such that the straight line $1/A=1/Qd$ is tangent to the bifurcation curve
$A=A_n^\R(d)$ for the smallest possible value of $T$.
\end{corol}

\section{Example}\label{sec:examples}
In this section we use the results presented so far to study the behavior of the
firing-rate under frequency variation for different configurations. For the sake
of simplicity, we choose to study such configurations for a linear system, as it
will permit us to compute explicitly the quantities involved in the results of
section~\ref{sec:freq_properties}. However, we emphasize that these quantities
are straight forward to compute numerically for other type of systems for which
conditions~\conds{} hold.
\subsection{Linear integrate and fire model}
Let
\begin{equation}
f(x)=ax+b.
\label{eq:linear_system}
\end{equation}
In order to satisfy conditions~\conds{}, we require that $a<0$ and
$\bx=-b/a\in(0,\theta)$, where $x=\theta$ is the threshold of the integrate and
fire system~\SYSTEMWR{}.\\

For system~\SYSTEMWRLinear{}, the critical dose~\eqref{eq:critical_dose_eq}
becomes
\begin{equation*}
Q_c=-(a\theta+b),
\end{equation*}
which is the minimal amplitude of the pulse~\eqref{eq:pulse} for which the
system~\SYSTEMWRLinear{} can exhibit spikes.

The linearity of the system permits us to also  explicitly compute the quantity
$\delta$ involved in Proposition~\ref{prop:fr_Tlarge},
\begin{equation}\label{eq:delta_linear}
\delta=\frac{1}{a}\ln\left( \frac{\theta a}{b+A}+1 \right).
\end{equation}
The averaged version of system~\SYSTEMWRLinear{} becomes $\dot{x}=f(x)+Ad$, for
which we can also explicitly compute the quantity $\hdelta$ involved in
Proposition~\ref{prop:fr_Tsmall},
\begin{equation*}
\hdelta=\frac{1}{a}\ln\left( \frac{\theta a}{b+Ad}+1 \right).
\end{equation*}
\begin{figure}
\begin{picture}(1,1.1)
\put(0,1.1){
\subfigure[\label{fig:T0d1}]{\includegraphics[angle=-90,width=0.5\textwidth]
{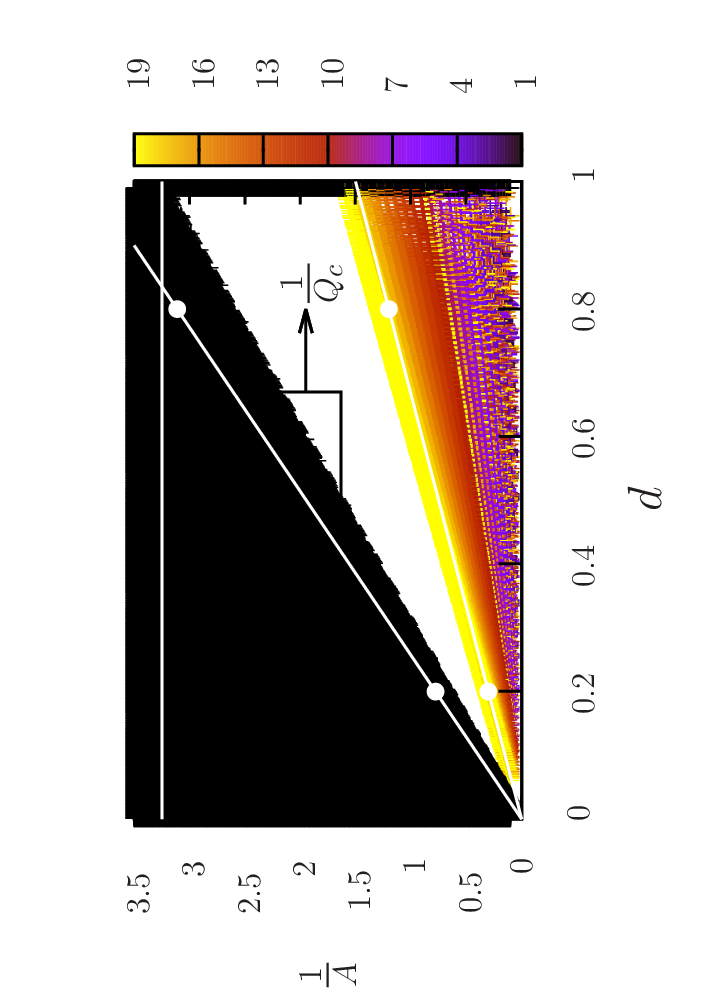}}
}
\put(0.5,1.1){
\subfigure[]{\includegraphics[angle=-90,width=0.5\textwidth]
{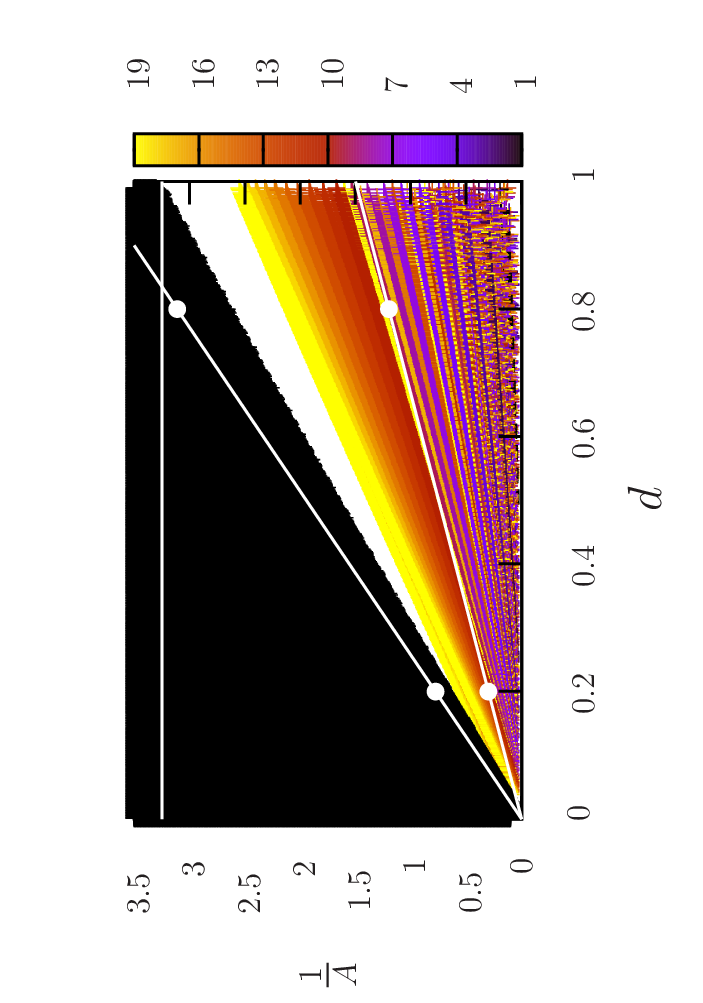}}
}
\put(0,0.75){
\subfigure[]{\includegraphics[angle=-90,width=0.5\textwidth]
{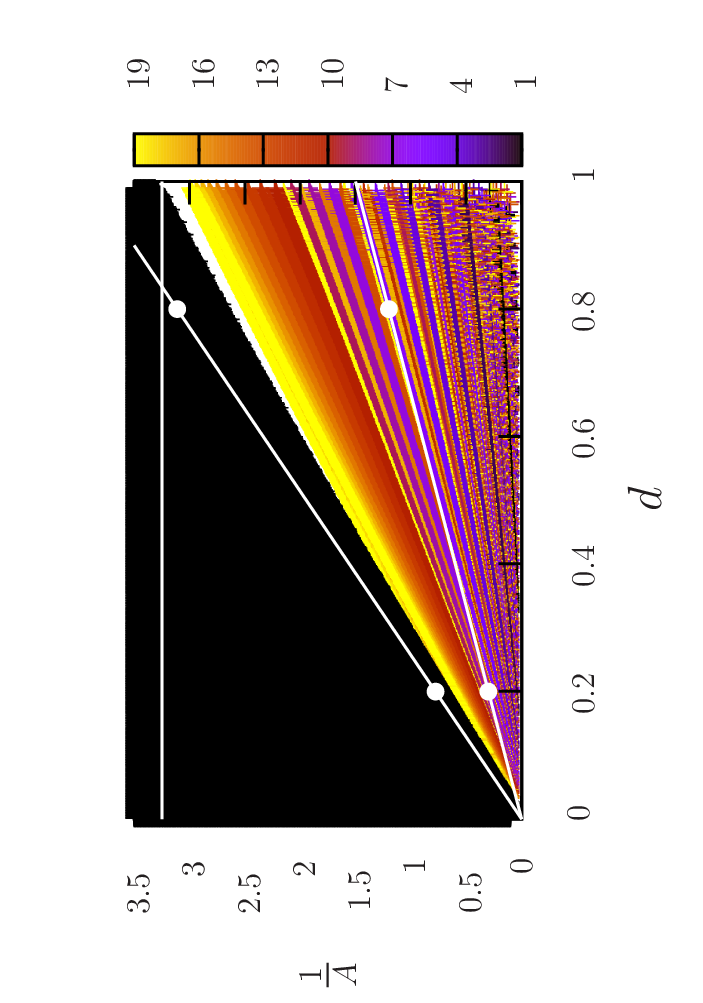}}
}
\put(0.5,0.75){
\subfigure[]{\includegraphics[angle=-90,width=0.5\textwidth]
{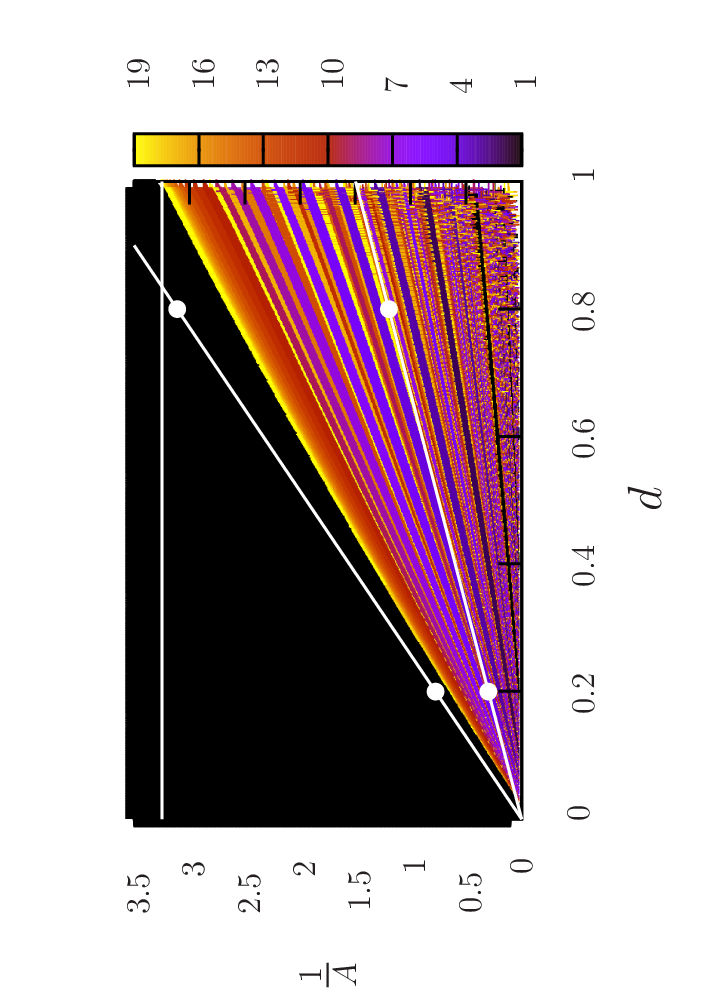}}
}
\put(0,0.4){
\subfigure[]{\includegraphics[angle=-90,width=0.5\textwidth]
{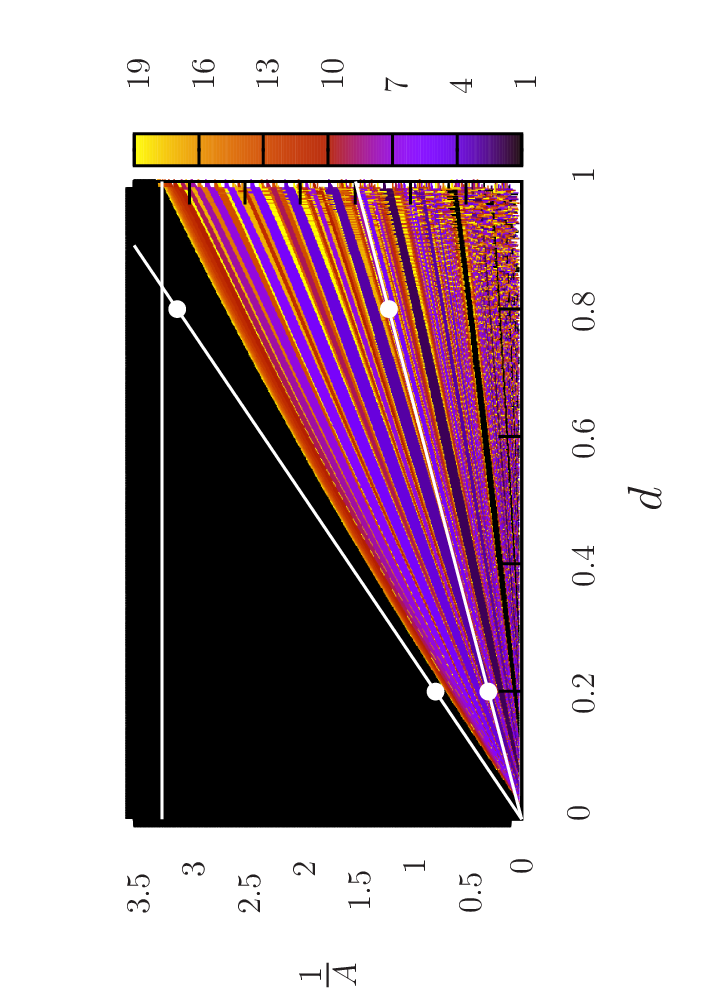}}
}
\put(0.5,0.4){
\subfigure[]{\includegraphics[angle=-90,width=0.5\textwidth]
{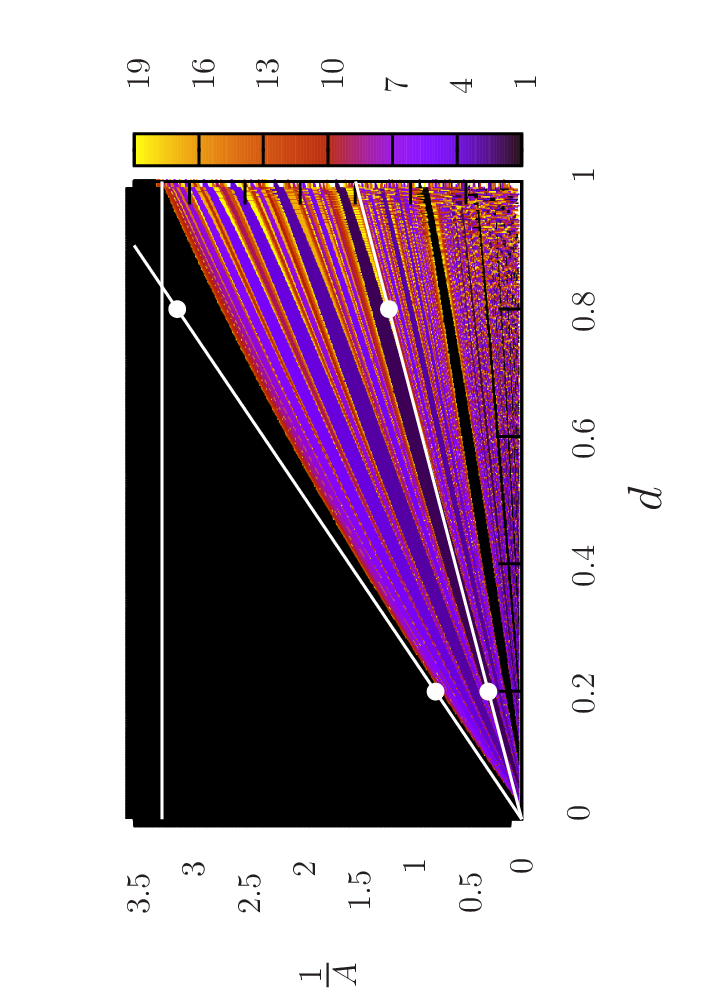}}
}
\end{picture}
\caption{Bifurcation scenarios for $T=0.1$ (a), $T=0.2$ (b), $T=0.3$ (c),
$T=0.5$ (d), $T=0.7$ (e) and $T=1$ (f), for $\theta=1$, $b=0.2$, $a=-0.5$. For
clarity reasons only periods lower than $20$ are shown; regions with higher
periods are filled in white. The horizontal line ($1/A=1/Q_c$) separates the
non-spiking and the spiking regions. The straight line with slope $1/Q_c$
labeled in (a) separates the permanent and conditional spiking
regions. The white lines are given by parameters leading to the
same dose: $Ad=\text{const.}$}
\label{fig:diff_Ts}
\end{figure}
\begin{figure}
\begin{picture}(1,1.1)
\put(0,1.1){
\subfigure[]{\includegraphics[angle=-90,width=0.5\textwidth]
{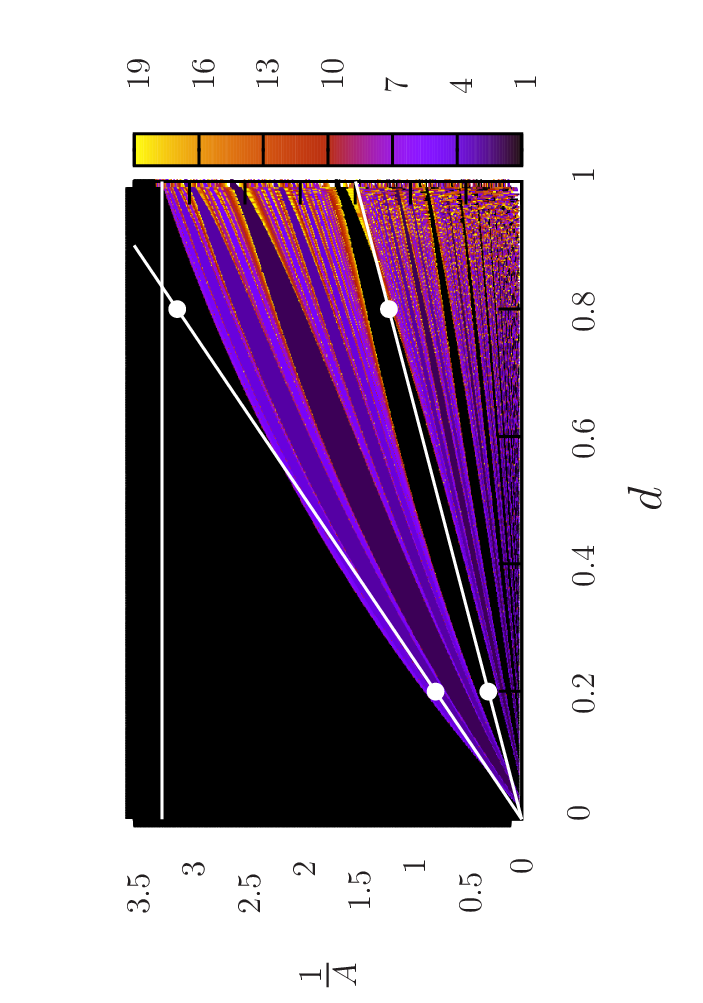}}
}
\put(0.5,1.1){
\subfigure[]{\includegraphics[angle=-90,width=0.5\textwidth]
{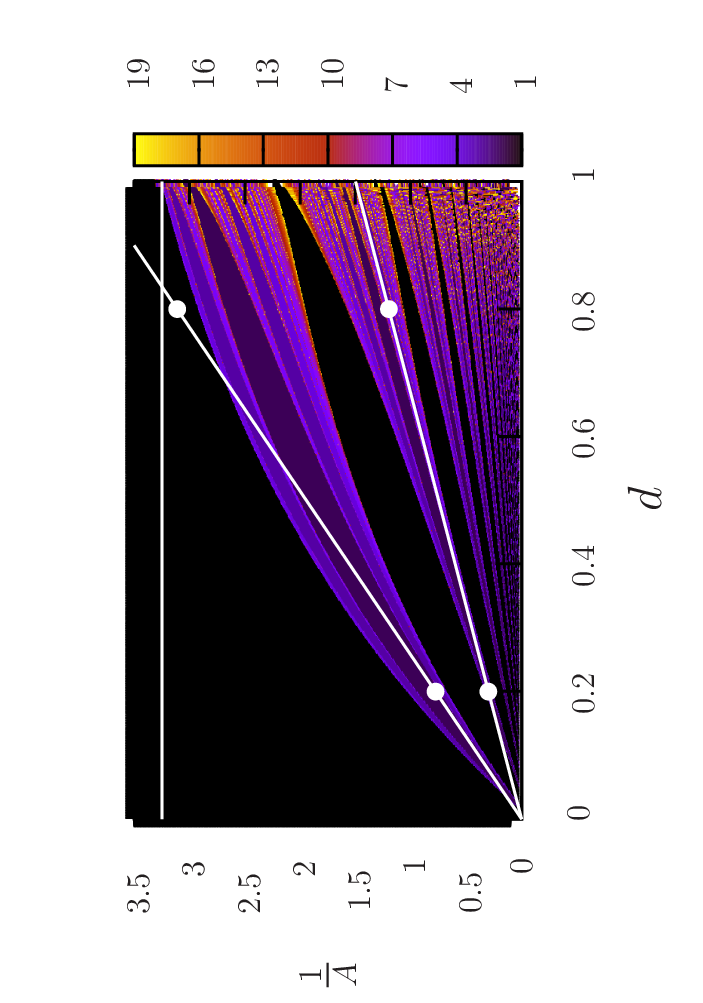}}
}
\put(0,0.75){
\subfigure[]{\includegraphics[angle=-90,width=0.5\textwidth]
{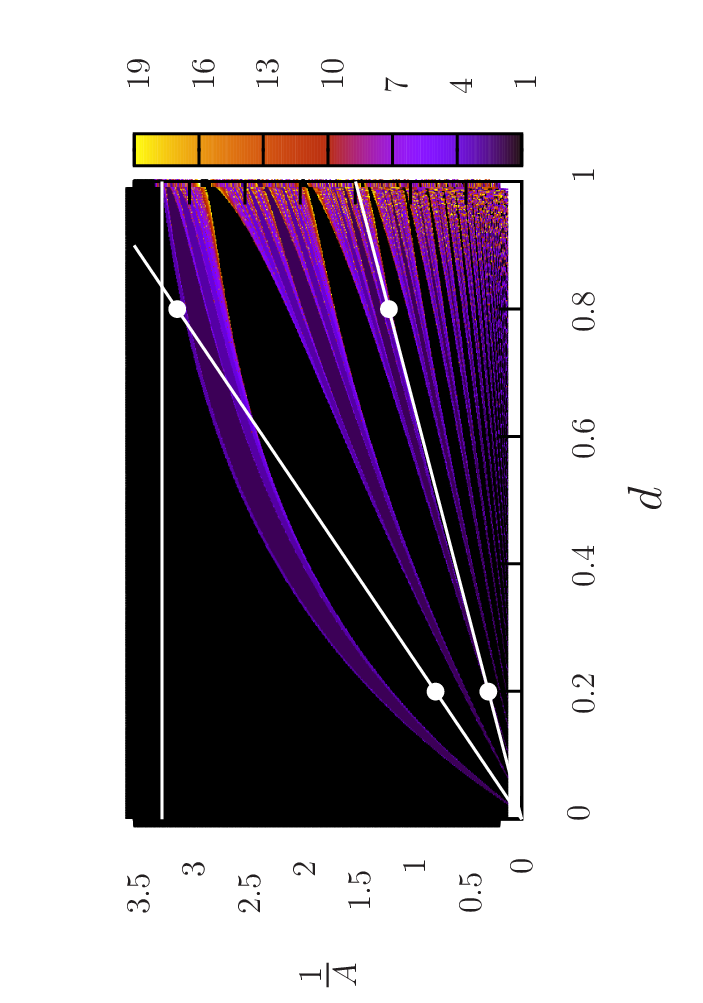}}
}
\put(0.5,0.75){
\subfigure[]{\includegraphics[angle=-90,width=0.5\textwidth]
{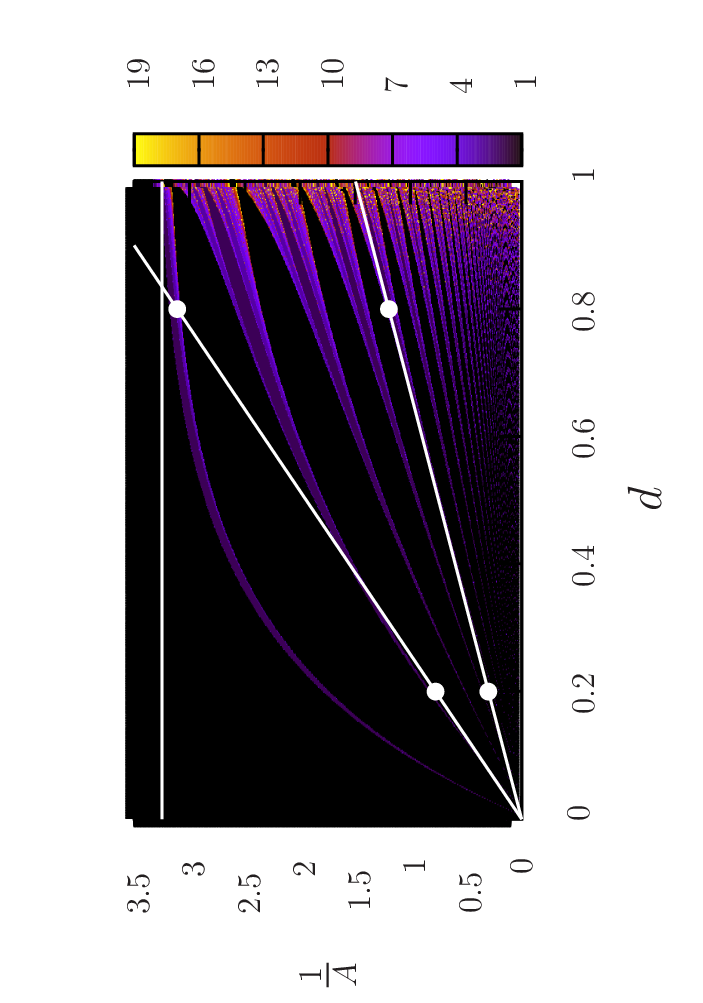}}
}
\put(0,0.4){
\subfigure[]{\includegraphics[angle=-90,width=0.5\textwidth]
{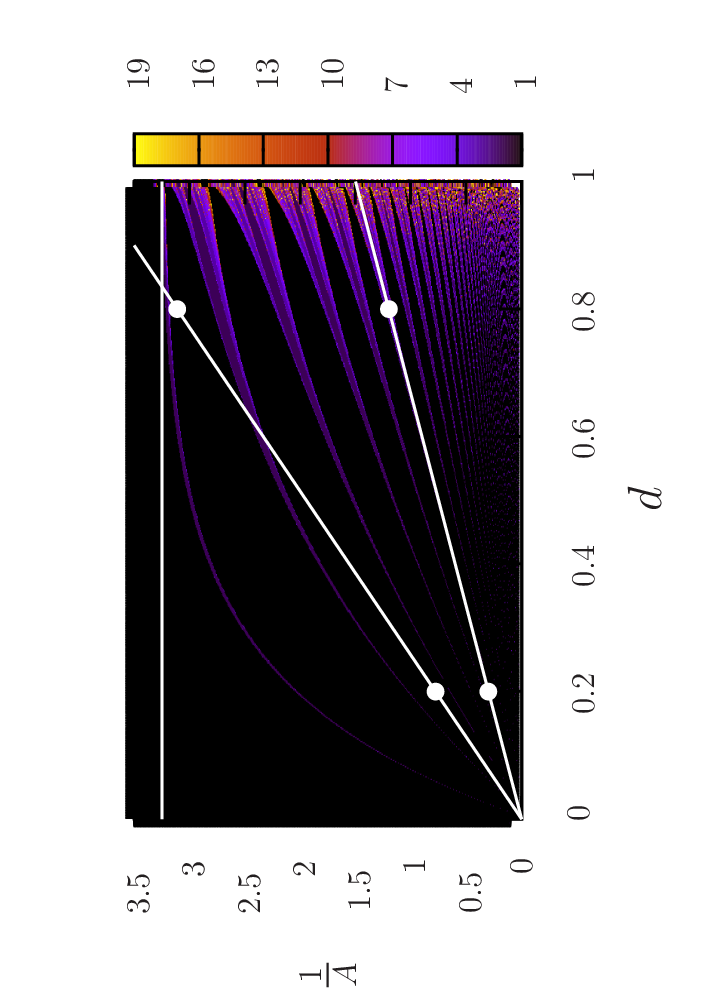}}
}
\put(0.5,0.4){
\subfigure[]{\includegraphics[angle=-90,width=0.5\textwidth]
{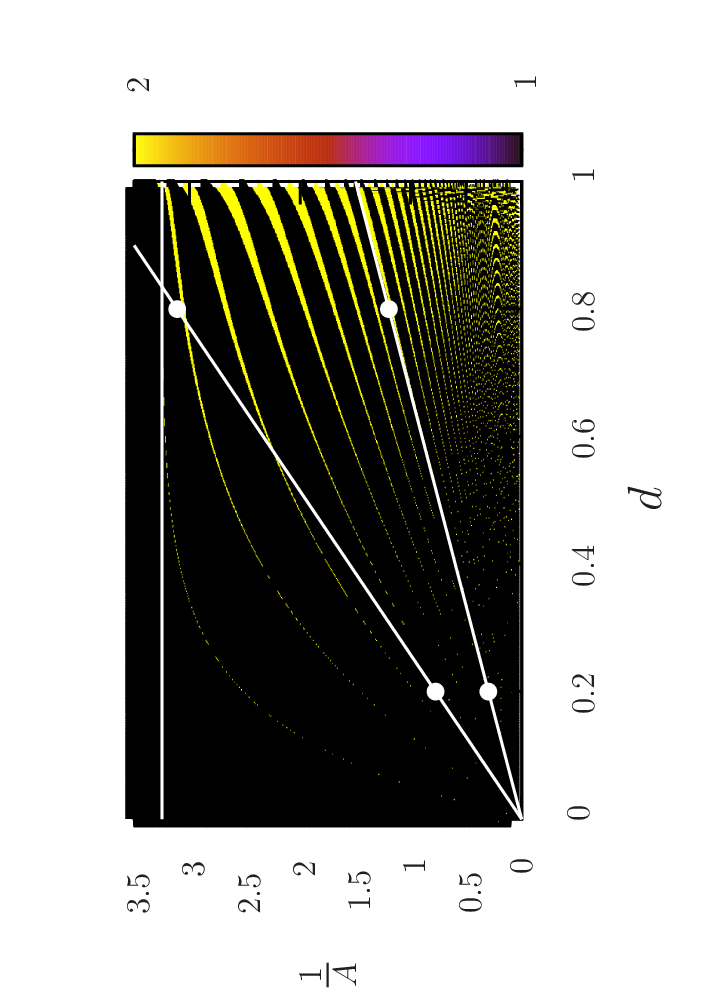}}
}
\end{picture}
\caption{Bifurcation scenarios for $T=2$ (a), $T=3$ (b), $T=5$ (c),
$T=8$ (d), $T=10$ (e) and $T=15$ (f). See caption in
Fig.~\ref{fig:diff_Ts} for more details. In (f) only periods up to
$2$ are shown in order to improve comparison between (e) and
(f).}
\label{fig:diff_largeTs}
\end{figure}

In Figures~\ref{fig:diff_Ts} and~\ref{fig:diff_largeTs} we show the
bifurcation scenario in the $d\times 1/A$ parameter space for
different values of $T$. As mentioned in
Remark~\ref{rem:contractiveness}, for some values of $T$ the
stroboscopic map $\s$ may lose contractiveness in the domain
$[\Sigma_n,\theta]$. When this occurs, the rotation number (and hence
the firing number and firing-rate) do not follow a devil's staircase
but a monotonically increasing continuous function.\\
As shown in~\Firstpaper{}, for the linear case the contracting
condition becomes
\begin{align*}
&F(A,T,d)=\left| [\Sigma_n,\theta] \right|-\left| \s\left(
[\Sigma_n,\theta]
\right) \right|=\theta-\Sigma_n-(\s(\theta)-\s(\Sigma_n))\\
&=\theta-\frac{b+A}{a}\left( e^{a\left( n\delta-dT
\right)}-e^{a\left( T-(n-1)\delta \right)}+e^{aT(1-d)}-1 \right).
\end{align*}
On one hand, one sees that $F$ is a monotonically decreasing function
of $A$. On the other hand, when the fixed point $\bx_{1}$ undergoes a
right border collision ($A=A_1^\R$) $\s$ is contractive in the whole
domain $[0,\theta]$. Therefore, if it exists, the region where the
rotation number is not a devil's staircase is bounded in the parameter
space $d\times1/A$ between the curves given by $A=A_0(d)$ and
$A=A_1^\R(d)$. Hence, below the curve given by $A=A_1^\R$ only devil's
staircases given by the period adding bifurcation structures can
exist. See~\Firstpaper{} for more details.

As predicted by Proposition~\ref{prop:Tsmall}, when $T\to0$ the first
bifurcation curve, $A_0(d)$, tends to be the straight line
$1/A=1/Q_cd$, and the rest of bifurcation curves accumulate at $1/A=0$
(see Figure~\ref{fig:diff_Ts}).\\
By contrast, when $T\to\infty$, all bifurcation curves accumulate at the
horizontal curve $1/A=1/Q_c$, as predicted by Proposition~\ref{prop:Tlarge}
(see Figure~\ref{fig:diff_largeTs}).

We are now interested in studying the firing-rate $r$~\eqref{eq:firing_rate}
under frequency variation. However, when varying the period of the
pulse~\eqref{eq:pulse}, we will restrict ourselves to pulses with constant
average (constant released dose or energy) $Q$ given in
equation~\eqref{eq:dose}.\\
Obviously, the output of the system will be sensitive to variations of the
injected energy (dose). Hence, in order to perform an analysis based exclusively
on frequency variation we will be interested in the variation of the frequency
of the stimulus while keeping the dose constant (dose conservation).

Note that points in the parameter space $d\times 1/A$ with a fixed dose are
located in the straight lines
\begin{equation*}
\frac{1}{A}=\frac{1}{Q}d.
\end{equation*}
In Figures~\ref{fig:diff_Ts} and~\ref{fig:diff_largeTs} we have highlighted
parameter values associated to two different doses. These are given by two
different white straight lines; the one with the larger slope ($Q<Q_c$) is fully
contained in the non-spiking region when $T$ small enough, while the other one
is contained in the spiking region for all values of $T$, and they will lead to
different qualitative responses.

Note that the dose conservation can be performed in three different ways in
order to keep the quantity $Q=Ad$ constant. In the first one one varies the
duration of the impulse $\Delta=dT$ as the period $T$ of the periodic input
$I(t)$ varies, while its amplitude $A$ is kept constant.  This is done by
keeping the duty cycle $d$ constant.\\
In the second one, the duration of the pulse is fixed, and one varies its
amplitude when $T$ is varied in order to keep constant the average of $I(t)$.\\
Of course, one can also simultaneously vary both magnitudes, giving rise to any
different types of parametrizations with respect to $T$ of the straight lines
corresponding to fixed dose.

In the next sections we separately study the first two cases.

\subsection{Fixed dose for constant impulse amplitude (width correction)}\label{sec:time_correction}
Taking into account that $1/A=d/Q$, for a fixed value of the amplitude
of the pulse it is enough to keep the duty cycle $d$ constant in order
to obtain an input with constant dose $Q$.  Hence, in this first
approach, we just fix one point in the parameter space $(d,1/A)$ and
vary $T$. This will allow us to directly apply the results shown in
\S~\ref{sec:freq_properties}.

In Figure~\ref{fig:freq_responseQ0d666} we focus on two points in the
parameter space located at the white straight line with lower slope in
Figs.~\ref{fig:diff_Ts} and \ref{fig:diff_largeTs} (higher dose,
$Q>Q_c$), and we show the firing-number, $\eta$, (left figures) and
the  firing-rate, $r$, (right figures) of the periodic orbits found
when varying $T$.\\
As announced in Corollary~\ref{cor:regions}, as $Q>Q_c$ these two points in the
parameter space are located in the permanent-spiking region and, hence, as
mentioned in Corollary~\ref{cor:eta_smallT}, the firing-number tends to
zero when $T\to0$. However, as predicted by Proposition~\ref{prop:fr_Tsmall},
the firing-rate fulfills
\begin{equation*}
\lim_{T\to0}r(T)=\frac{1}{\hdelta},
\end{equation*}
with $1/\hdelta=0.58$ for the used parameter values. As noted in
Remark~\ref{rem:r_depdends_on_Q}, this value only depends on $Q$ and hence it is
the same for all points with equal dose.

In Figure~\ref{fig:magni} we show a magnification of the firing
rate for small values of $T$, where one can clearly see the structure given by
the devil's staircase.

On the other hand, Proposition~\ref{prop:fr_Tlarge} provides the limiting value
for the firing-rate, 
\begin{equation*}
\lim_{T\to\infty}r(T)=\frac{d}{\delta},
\end{equation*}
where $\delta$ is given in~\eqref{eq:delta_linear}.
Note that this quantity depends on $A$ and, hence, it is different for the two
considered case although they correspond to inputs with the same average. For
$(d,1/A)=(0.2,0.3)$ (Figure~\ref{fig:freq_responseQ0d666} (b)) we get
$d/\delta=0.655$, and for $(d,1/A)=(0.8,1.2)$ we obtain $d/\delta=0.604$.

Finally, observe that the firing-rate possesses a global maximum and minimum at
 $T=T_1^\R$ and $T=T_1^\LL$, respectively, as $\bx$ is attracting enough.\\

\begin{figure} \begin{picture}(1,0.8)
\put(0,0.8){
\subfigure[]{\includegraphics[angle=-90,width=0.5\textwidth]
{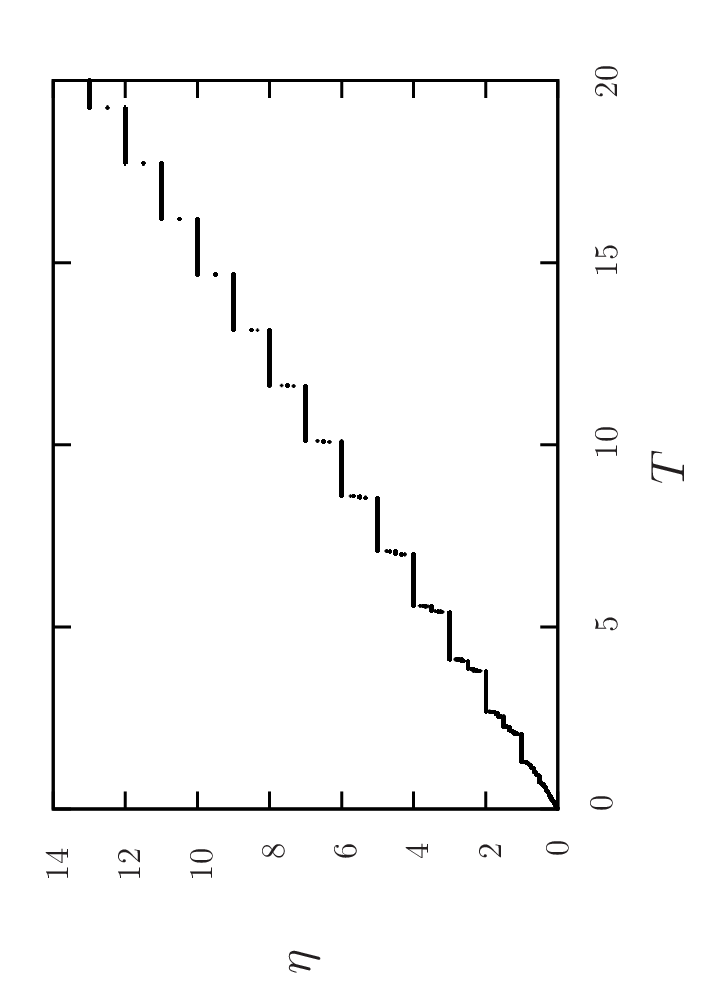}}
}
\put(0.5,0.8){
\subfigure[]{\includegraphics[angle=-90,width=0.5\textwidth]
{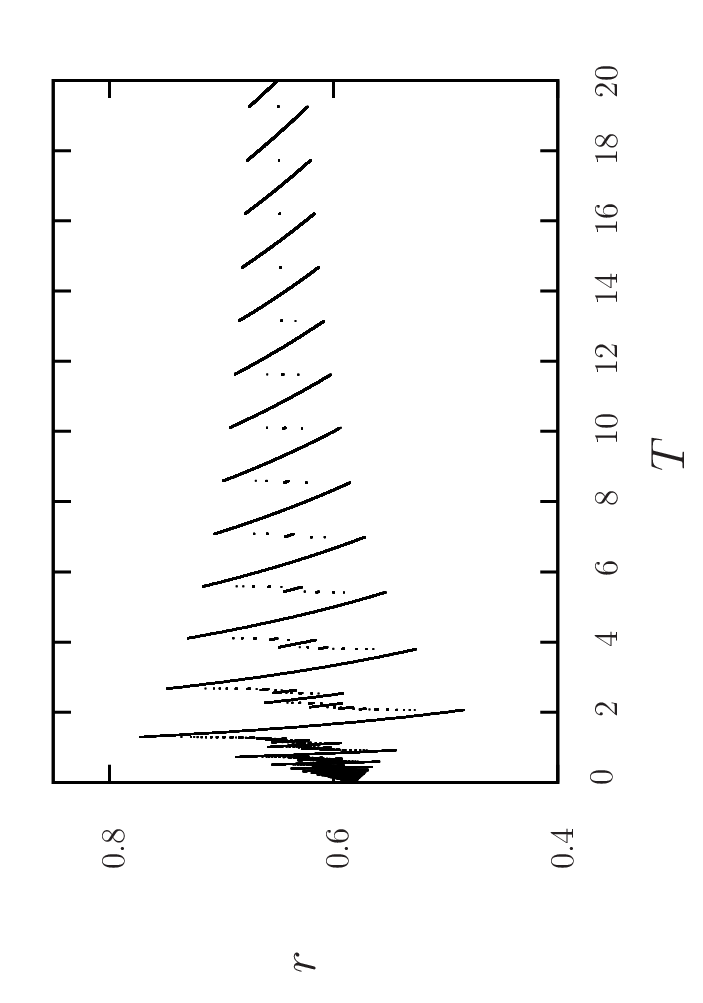}}
}
\put(0,0.4){
\subfigure[]{\includegraphics[angle=-90,width=0.5\textwidth]
{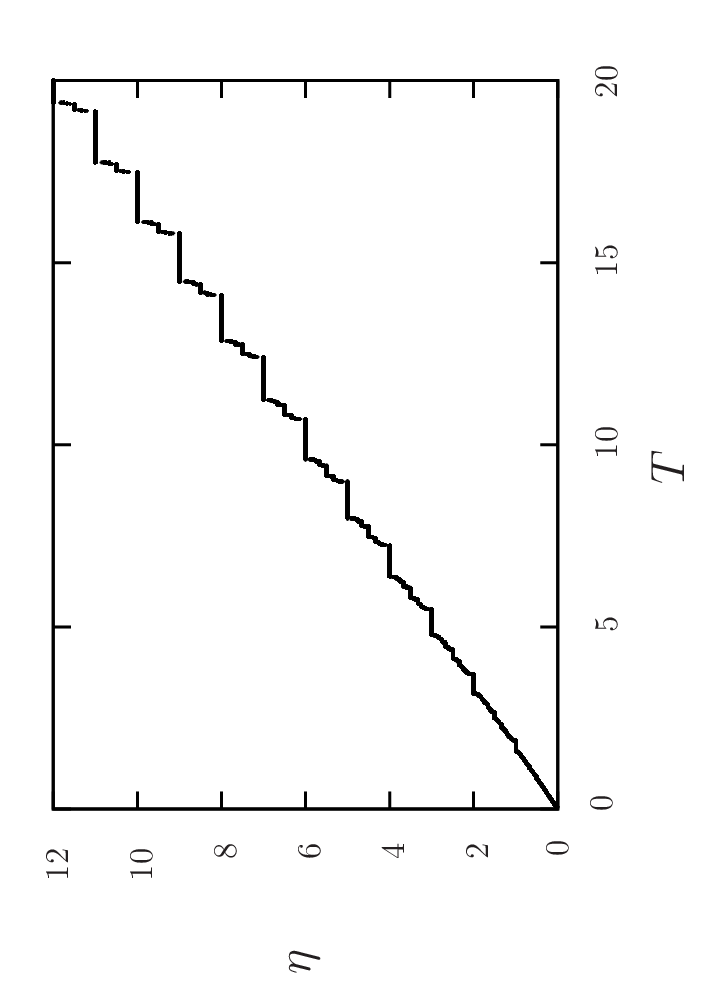}}
}
\put(0.5,0.4){
\subfigure[]{\includegraphics[angle=-90,width=0.5\textwidth]
{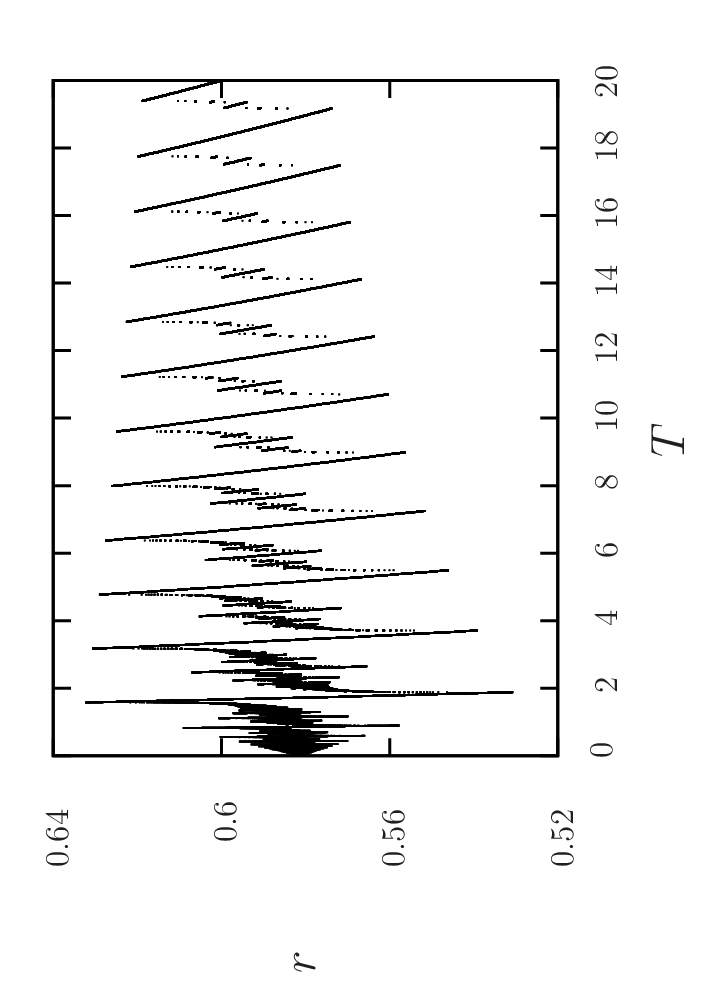}}
}
\end{picture}
\caption{Firing-number $\eta(T)$ (left) and firing-rate $r=\eta(T)/T$
(right), under variation of $T$ for $Q=0.666$. (a) and (b) $d=0.2$ and
$1/A=0.3$. (c) and (d) $d=0.8$ and $1/A=1.2$. The firing-rate follows a devil's
staircase with monotonically decreasing pieces exhibiting a maximum at
$T=T_1^\R$ and a minimum at $T=T_1^\LL$. Parameters $a$, $b$ and $\theta$ are
set as in Figure~\ref{fig:diff_Ts}.}
\label{fig:freq_responseQ0d666}
\end{figure}

\begin{figure}
\begin{center}
\begin{picture}(1,0.4)
\put(0,0.4){
\subfigure[]{\includegraphics[angle=-90,width=0.5\textwidth]
{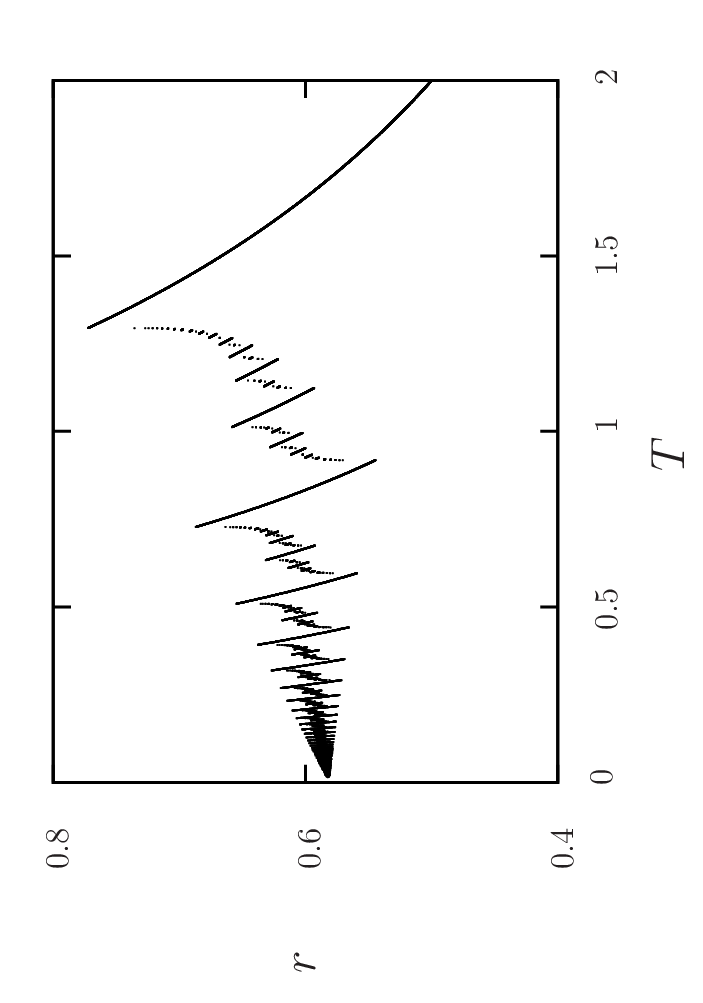}}
}
\put(0.5,0.4){
\subfigure[]{\includegraphics[angle=-90,width=0.5\textwidth]
{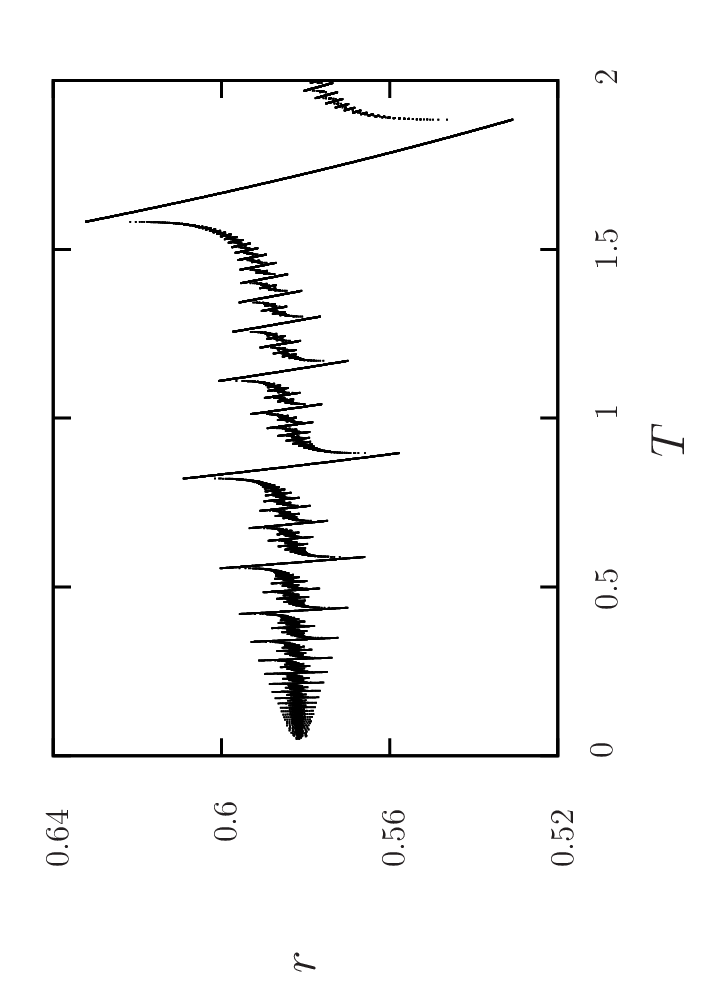}}
}
\end{picture}
\end{center}
\caption{Magnification of Figures~\ref{fig:freq_responseQ0d666} (b) and
(d).}\label{fig:magni}
\end{figure}

We now focus on two different inputs with average lower than the critical dose.
In Figure~\ref{fig:freq_responseQ0d257} we show the same results for the two
points labeled in Figures~\ref{fig:diff_Ts} and~\ref{fig:diff_largeTs} located
on the white straight line with higher slope (lower dose). For large values of
$T$ the firing-rate shows the same behavior as before with limiting values
$d/\delta=0.244$ (Figure~\ref{fig:freq_responseQ0d257} (b)) and $d/\delta=0.125$
(Figure~\ref{fig:freq_responseQ0d257} (d)). However, as predicted in
Corollary~\ref{cor:eta_smallT}, unlike in the previous case, as these two points
are now located in the conditional-spiking region, there exists some values of
$T$ below which the firing-rates vanish.

Note that, as in the previous case, the firing-rate exhibits a global maximum at
$T=T_1^\R$. However, the global minimum becomes now $0$ for all $0<T<T_0$, as
$(d,A)$ belongs to the conditional spiking region (see
Remark~\ref{rem:global_minimum}).

\begin{figure}
\begin{picture}(1,0.8)
\put(0,0.8){
\subfigure[]{\includegraphics[angle=-90,width=0.5\textwidth]
{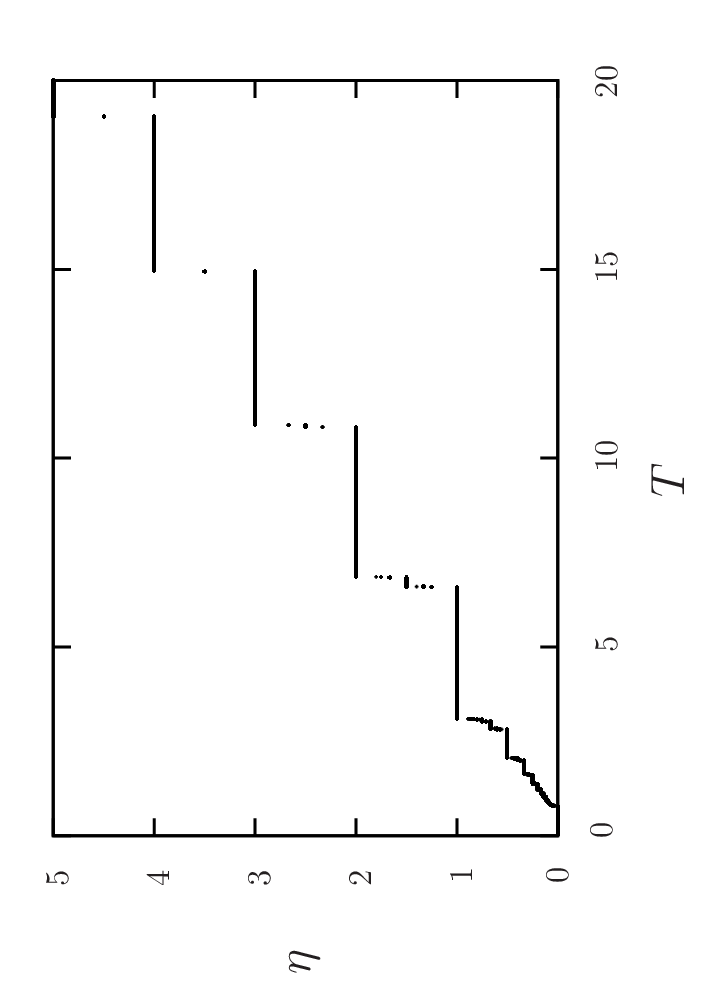}}
}
\put(0.5,0.8){
\subfigure[]{\includegraphics[angle=-90,width=0.5\textwidth]
{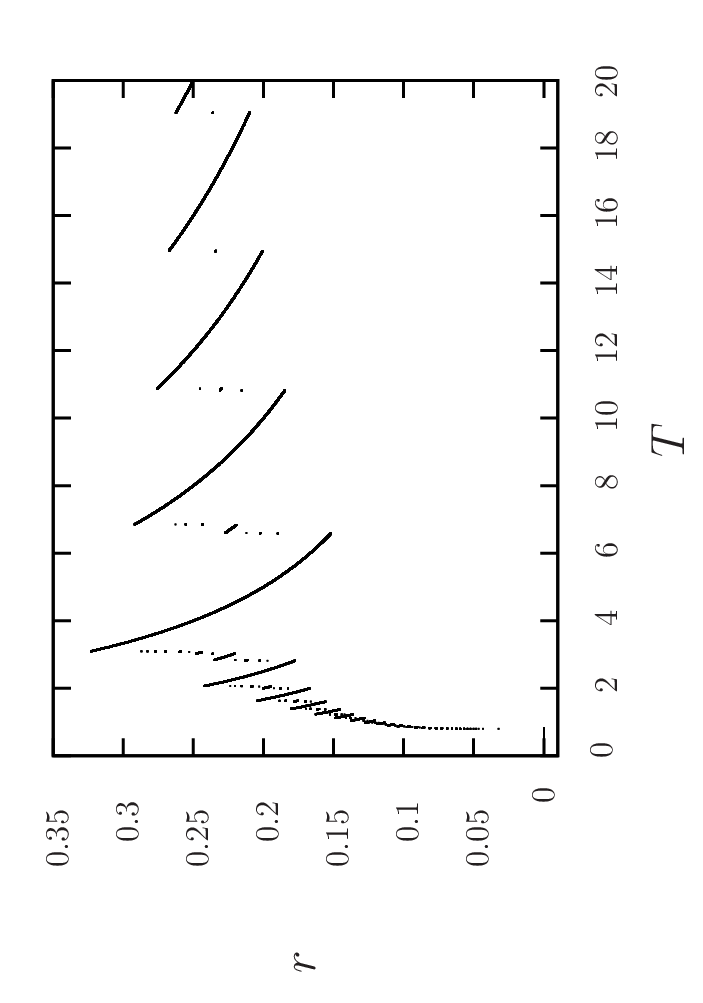}}
}
\put(0,0.4){
\subfigure[]{\includegraphics[angle=-90,width=0.5\textwidth]
{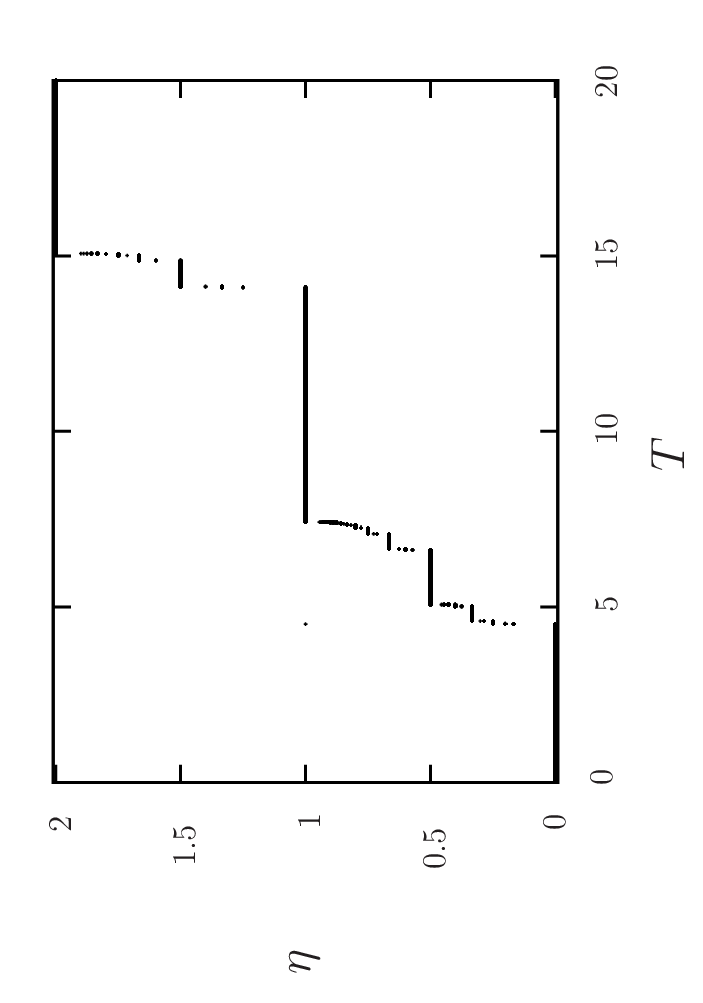}}
}
\put(0.5,0.4){
\subfigure[]{\includegraphics[angle=-90,width=0.5\textwidth]
{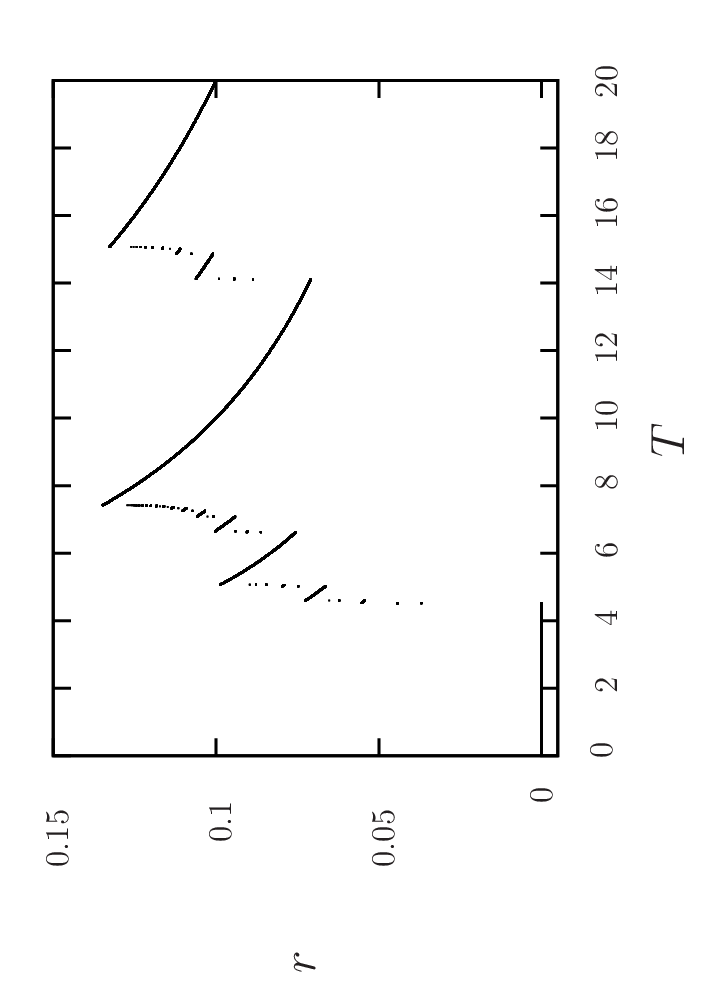}}
}
\end{picture}
\caption{Same as Fig.~\ref{fig:freq_responseQ0d666} for $Q=0.257$. (a) and (b):
$d=0.2$ and $1/A=0.777$. (c) and (d): $d=0.8$ $1/A=3.111$. As in the previous
case, a global maximum occurs at $T=T_1^\R$; however, the global minimum is
$0$ for a whole range of small periods.}
\label{fig:freq_responseQ0d257}
\end{figure}

\subsection{Fixed dose for fixed pulse duration (amplitude correction)}\label{sec:amplitude_correction}
We now fix the duration of the pulse $\Delta$ and perform the dose conservation
by properly modifying its amplitude. This is done by varying the parameters $d$
and $A$ along straight lines in the parameter space $d\times1/A$
parametrized by $T$,
\begin{equation}
\left(d,\frac{1}{A}\right)=\left(\frac{\Delta}{T},\frac{\Delta}{QT}\right).
\label{eq:straight_lines}
\end{equation}
Note that, with this approach, it is not possible to analyze the
properties of the output when $T\to0$, since its minimal value is
$T=\Delta$. Varying $T$ from $\Delta$ to $\infty$, one has to  vary
$(d,1/A)$ from $(1,1/Q)$ to $(0,0)$ along a straight line with slope
$1/Q$ in order to keep the released dose constant.

Regarding the behavior of $r(T)$ when $T\to\infty$, we can use
Proposition~\ref{prop:fr_Tlarge}. From~\eqref{eq:straight_lines} we get
$d=\Delta/T$ and $A=QT/\Delta$, which, when combined
with~\eqref{eq:delta_linear} and Proposition~\ref{prop:fr_Tlarge} gives us
\begin{equation*}
\lim_{T\to\infty}r(T)=\frac{Q}{\theta},
\end{equation*}
independently of $\Delta$.\\

In Figure~\ref{fig:freq_responseQ0d666_Delta3} we show the evolution of the
firing-rate for an input with average greater than the
critical dose, $Q_c$. Note that this leads to a broken devil's staircase, as it
starts at $T=\Delta$. The behavior at $T\to\infty$ is the expected one.

In Figure~\ref{fig:freq_responseQ0d257_Delta3} we show the same computation for
a $Q<Q_c$. In this case, if $T$ is close enough to $\Delta$,
equation~\eqref{eq:straight_lines} provides points located in the non-spiking
region for which $r(T)=\eta(T)=0$. Although the points provided
by~\eqref{eq:straight_lines} are never located in the permanent-spiking region,
they are in the conditional-spiking region if $T$ is large enough. Hence, one
starts observing spikes at some point.

\begin{figure}
\begin{picture}(1,0.4)
\put(0,0.4){
\subfigure[]{\includegraphics[angle=-90,width=0.5\textwidth]
{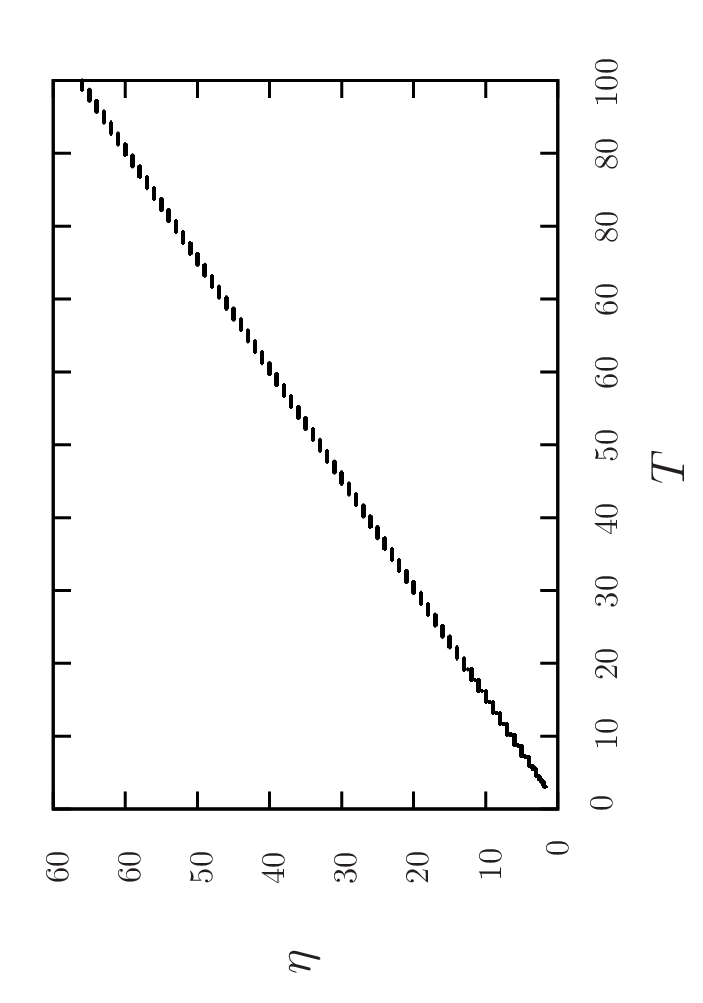}}
}
\put(0.5,0.4){
\subfigure[]{\includegraphics[angle=-90,width=0.5\textwidth]
{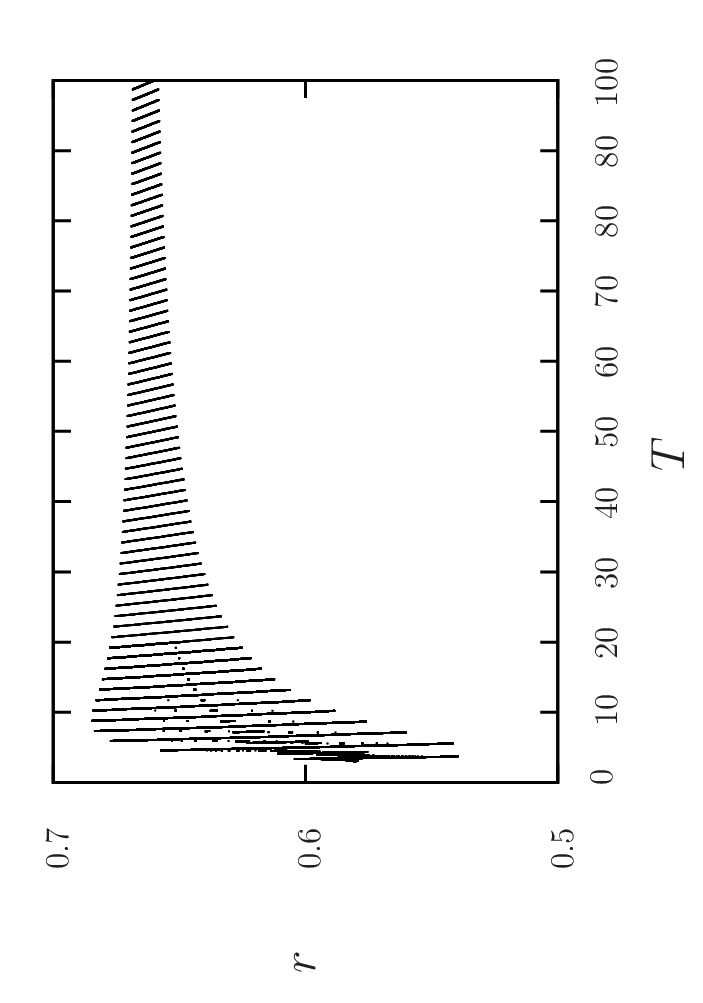}}
}
\end{picture}
\caption{Firing-number (a) and firing-rate (b) for $Q=0.666$ using amplitude
correction while keeping constant the duration of the pulse, $\Delta=3$.}
\label{fig:freq_responseQ0d666_Delta3}
\end{figure}
\begin{figure}
\begin{picture}(1,0.4)
\put(0,0.4){
\subfigure[]{\includegraphics[angle=-90,width=0.5\textwidth]
{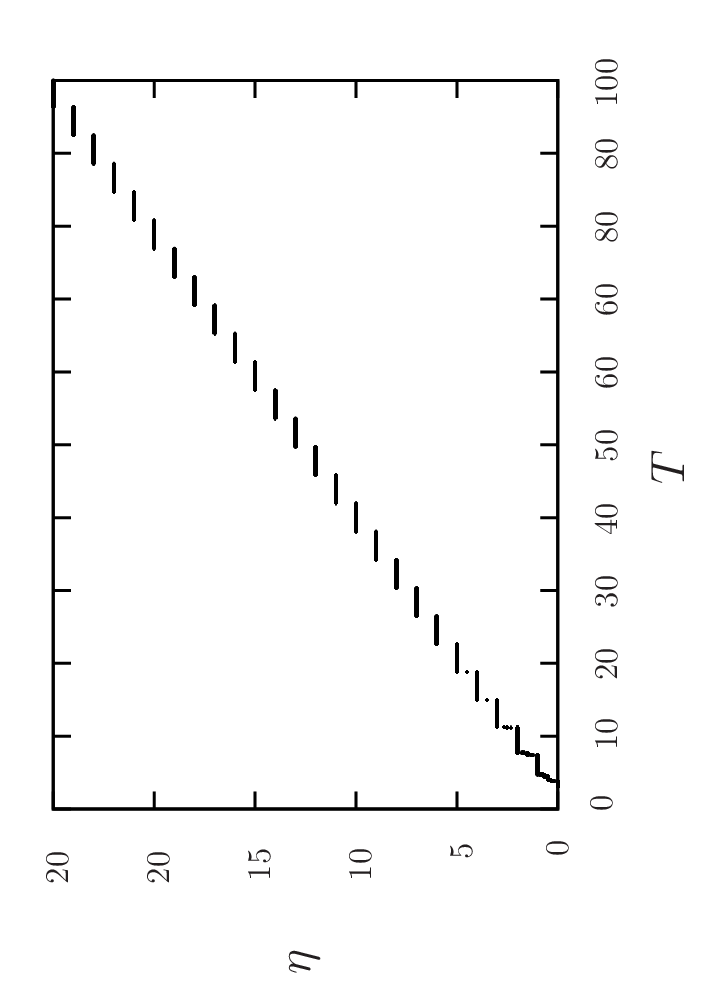}}
}
\put(0.5,0.4){
\subfigure[]{\includegraphics[angle=-90,width=0.5\textwidth]
{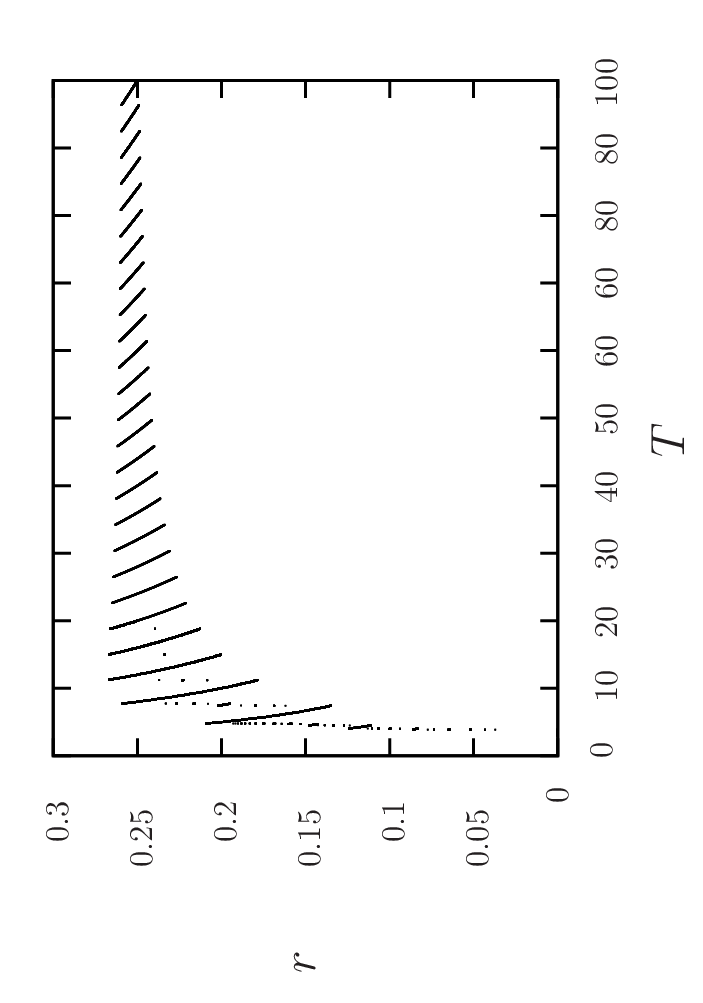}}
}
\end{picture}
\caption{Firing-number (a) and firing-rate (b) for $Q=0.257$ using amplitude
correction while keeping constant the duration of the pulse, $\Delta=3$.}
\label{fig:freq_responseQ0d257_Delta3}
\end{figure}

Note that, unlike when fixing the dose by width correction, one always gets non
zero spikes per period and non-zero firing-rates, at least for small enough
frequencies. This is because, when following the straight
lines~\eqref{eq:straight_lines} towards the origin one always enters the spiking
regions.

\section{Conclusions}
In this paper we have considered a generic spiking model
(integrate-and-fire-like system) with an attracting equilibrium point in the
subthreshold regime forced by means of a pulsatile (square wave) periodic input.
By contrast to the usual approach~\cite{KeeHopRin81,CooOweSmi01,TouBre08},
we consider the stroboscopic map instead of the Poincar\'e map onto the threshold.
This Poincar\'e map becomes a regular map there where it is defined while the
stroboscopic map is discontinuous. However, as shown in~\Firstpaper{}, this
becomes indeed an advantage, as this type of maps are well understood map by the
piecewise-smooth community (see~\cite{AlsGamGraKru14} for a recent survey).\\
As shown in~\cite{GraKruCle13} the system exhibits spiking dynamics organized in
rich bifurcation structures in the parameter space formed by the amplitude and
duty cycle of the forcing pulse. These bifurcations and the associated
symbolic dynamics completely explain relevant features of this type of
excitable systems, like the firing-rate. In this work, we have studied how these
bifurcation structures, and dynamical properties associated with them,
vary when the period of the forcing is varied while keeping the injected dose
(input average) constant.  We have given special interest to the asymptotic
firing-rate (average number of spikes per unit time), which turns out to follow
a devil's staircase (a fractal structure) with monotonically decreasing steps.
In particular, we have precisely characterized its global maximum in the whole
frequency domain as well each local maxima.
%
If we consider specific ranges of frequency whose bounds correspond to
the frequencies eliciting the subsequent local minima, the response
can be decomposed in a repetitive structure with a non-monotonic,
bell-shaped pattern and global maximum.

\newcommand{\etalchar}[1]{$^{#1}$}
\def\zh{Zh}\def\yu{Yu}\def\ya{Ya}

\end{document}